\documentclass[12pt]{article}
\usepackage{graphics}%%% option packages
\usepackage[dvips]{graphicx}%%% option packages

\usepackage{amsmath,amssymb,amsthm}
\usepackage[utf8]{inputenc}
\usepackage{tikz}
\usepackage{subfigure}
\usepackage{comment}

\title{
Billiards in a circle with trajectories circumscribing  a triangle
}

\author{Takeo Noda and Shin-ichi Yasutomi}

\newtheorem{thm}{Theorem}[section]

\newtheorem{lem}[thm]{Lemma}
\newtheorem{prop}[thm]{Proposition}

\theoremstyle{definition}

\newtheorem{example}{Example}[section]

\theoremstyle{remark}
\newtheorem{rem}{Remark}[section]

\numberwithin{equation}{section}

\usepackage{color}

\begin{document}

\maketitle

\begin{abstract}
Dogru and Tabachnikov in 2003 explored the polygonal outer billiard map in the hyperbolic plane and introduced a class of convex polygons called 'large'. They particularly sought conditions for a triangle to be classified as large.
For a large triangle, there exist  two triangles that are circumscribed around it and inscribed within the unit circle.
In the Klein-Beltrami model of hyperbolic geometry, we reformulate the conditions for a triangle to be classified as 'large'  in a more Euclidean geometric manner. A proposed measure of its Euclidean geometric size when the triangle is considered 'large' is introduced, and an evaluation of this measure is conducted.
We also provide an explicit formula for an isosceles triangle.
\end{abstract}

\section{Introduction}
Dogru and Tabachnikov \cite{DT} explored the polygonal outer billiards map in the hyperbolic plane.
The outer billiard is also known as the dual billiard.
Concerning the dual billiard, refer to Chapter 9 in \cite{T} for further details.
Let $C$ be a convex polygon in the hyperbolic plane.
Generally, from a point $v$ outside of $C$, two tangent lines can be drawn, intersecting $C$ at points $u_1$ and $u_2$ respectively.
%Let's
Let us assume that the direction of %$vu_1$
$vu_2$ is in a counterclockwise direction compared to the direction of %$vu_2$
$vu_1$, with an angle of less than $\pi$.
We define $F(v)$ as the point $w$ on the extension of $vu_1$ such that $|vu_1| = |u_1w|$.
By considering limits, the map $F$ naturally extends to a set of infinity points and is denoted as $f$ when restricted to this set of infinity points.
Let $\rho(C) \in (0, 1/2)$ be the rotation number of $f$.
In \cite{DT} a class of $n$-gons ($n\geq 3$) is identified for
which the rotation number equals to $1/n$ and $f$ 
has a hyperbolic $n$-periodic orbit; 
these polygons are called “large”. They provided a condition for $C$ to be large, especially when it is a triangle.

\begin{thm}[Dogru and Tabachnikov \cite{DT}]\label{Dogru and Tabachnikov}
    $\triangle P_1P_2P_3$ is large if and only if $H > 1$ where $H$ is given by any of the following equal expressions: for $i=1,2,3$
\begin{align*}
H&=\sinh h_i \sinh a_i=\sin \alpha_i \sinh a_{i+1} \sinh a_{i+2}=2\sinh s \tanh r=\\
&=2\sqrt{\sinh  s \sinh (s-a_1) \sinh (s-a_2)  \sinh (s-a_3)}=\\
&=4\sin(K/2)\cosh(a_1/2)\cosh(a_2/2)\cosh(a_3/2),
\end{align*}
where  $a_i$ is the length of the side corresponding to the point $P_i$,
$s=(a_1+a_2+a_3)/2$, $\alpha_i$ is the angle of $P_i$, $h_i$ is the altitude
dropped on the $i$-th side, $K$ is  the area, $r$ is the radius of the inscribed circle and the index $i$ is cyclic with respect to $\mod 3$.
\end{thm}

Dogru and Tabachnikov gave the following Theorem.
\begin{thm}[Dogru and Tabachnikov \cite{DT}]
    If $C$ is a large polygon then all orbits of the dual billiard map escape to infinity.
\end{thm}

We remark that in Theorem \ref{Dogru and Tabachnikov} the condition for the rotation number to equal $1/3$ is  $H \geq 1$, with equality holding when
$f$ has a unique $3$-periodic orbit.
%Therefore, in this paper, when $C$ is a triangle and $\rho(C)=1/3$, we will refer to $C$ as {\it large}.
In \cite{DT}, the Klein-Beltrami model of hyperbolic geometry is primarily used.
In this model, line segments are  %line segments 
those 
in Euclidean geometry. Furthermore, when the triangle $C$ is considered large and positioned near the center of the circle, the theorem implies a certain minimum size requirement. This seems to be the origin of the term %large.
``large".

However, concerning $f$, it is a transformation on $S^1$ and can be naturally interpreted as a conventional geometric map, departing from hyperbolic geometry.
For example, let $C$ be a circle lying inside ${S^1}$. 
The trajectory related to $(f^n(v)(=v_n)$ for $n=0,1,2,\ldots$ (see Figure \ref{fig:introduction}(a)) has fascinated many researchers, since  the finding of Poncelet porism (cf. \cite{DR},\cite{T}).
Even though the rotation number is rational,
Poncelet porism states that 
$f$ 
is  conjugated to the rotation.
In a generalized situation, 
Mozgawa \cite{M} considered such a billiard problem
and obtained  a theorem 
similar to  Poncelet porism by using two ovals instead of 
circles.
Mozgawa et al. \cite{CMM}, \cite{M} called such a billiard problem a bar billiard.
Cima et al. \cite{CGM} considered   
a bar billiard problem, where the unit circle 
is placed inside the the curve $\{x^{2m}+y^{2m}=2\},\ m\in {\mathbb{Z}_{>0}}$, and
showed the rotation 
number
is $\frac{1}{4}$ and
the map associated with the billiard problem is not
conjugated to a rotation except for $m=1$.
Theorem \ref{Dogru and Tabachnikov} can also be interpreted as addressing the bar billiard problem related to a circle and a triangle and it demonstrates intriguing properties related to circles and triangles; i.e.,
for a large triangle, there exist  two triangles that are circumscribed around it and inscribed within the circle.

For $P_1, P_2$ in the unit circle, the set of $P_3$ such that the triangle $\triangle P_1P_2P_3$ is large lies on or outside the circumference of an ellipse (see Figure \ref{fig:introduction2}).
We provide a  description of this ellipse as a reformulation of Theorem \ref{Dogru and Tabachnikov}.
As an application of it, 
for large triangles, their size are measured from the perspective of standard Euclidean geometry, and their characteristics are investigated.
The following assumes quantities in Euclidean geometry, such as length and angle, and also includes congruence and similarity.
For a triangle $C$, define $\kappa(C)$ as follows: if $C$ is an acute-angled triangle, let it be the radius of its circumcircle; if $C$ is an obtuse-angled triangle or an right triangle, let it be half the length of its longest side.
Let $\mu(C)$ denote the upper bound of $\kappa(C')$ for triangles $C'$ similar to $C$ and satisfying the condition of being not large.
Then, if $\kappa(C)>\mu(C)$, 
$C$ is large.
We establish $\mu(C)\geq 1/2$ with equality holding only when it is an equilateral triangle.
We also provide an explicit formula of $\mu(C)$ for an isosceles triangle $C$.

In the context of the calculations, the manipulation of equations was facilitated using Sage Mathematics Software \cite{Sage}.

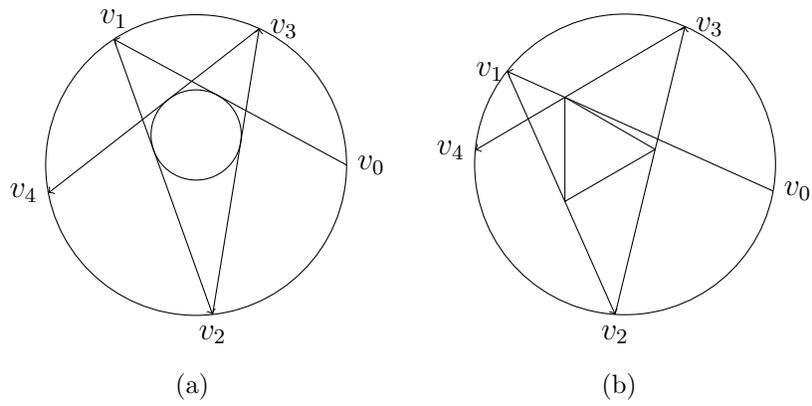
\begin{figure}
    \centering
    \subfigure[]{
    \begin{tikzpicture}
\draw(0,0) circle (2);
\draw(0,0.4) circle (0.6);
\draw[->] (2*1, 0)--(-0.547039*2, 0.837107*2);
\draw[->] (-0.547039*2, 0.837107*2)--(0.109269*2, 2*-0.994012);
\draw[->] (0.109269*2, 2*-0.994012)--(0.419377*2, 0.907812*2);
\draw[->] (0.419377*2, 0.907812*2)--(-0.982513*2, -0.186196*2);
);
 \draw (2*1, 0)node[right]{$v_0$};
 \draw (-0.547039*2, 0.837107*2)node[above]{$v_1$};
\draw (0.109269*2, 2*-0.994012)node[below]{$v_2$};
\draw (0.419377*2, 0.907812*2)node[right]{$v_3$};
\draw (-0.982513*2, -0.186196*2)node[left]{$v_4$};
    
    \end{tikzpicture}
    }
    \subfigure[]{
    \begin{tikzpicture}
\draw(0,0) circle (2);
\draw [](0.2*2, 0.1*2)--(-0.4*2, 0.44641*2)--(-0.4*2, -0.24641*2)--cycle;
\draw[->] (0.398194*2, 2*0.917301)--(-0.995466*2, 0.0951178*2);
\draw[->] (-0.0662161*2, -0.997805*2)--(0.398194*2, 2*0.917301);
\draw[->] (-0.784777*2, 0.619778*2)--(-0.0662161*2, -0.997805*2);
\draw[->] (0.984166*2, -0.177249*2)--(-0.784777*2, 0.619778*2);
);

 \draw (0.984166*2, -0.177249*2)node[right]{$v_0$};

 \draw (-0.784777*2+0.1, 0.619778*2)node[left]{$v_1$};
\draw (-0.0662161*2, -0.997805*2)node[below]{$v_2$};
\draw (0.398194*2, 2*0.917301)node[right]{$v_3$};
\draw (-0.995466*2, 0.0951178*2)node[left]{$v_4$};
    \end{tikzpicture}
    }
    
    \caption{Trajectories circumscribing  a circle or a triangle.}
    \label{fig:introduction}
\end{figure}

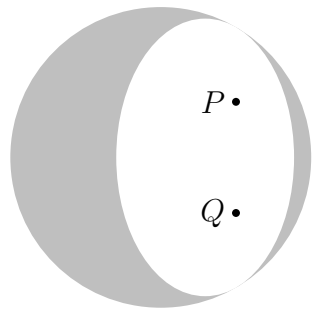
\begin{figure}
    \centering
    \begin{tikzpicture}
\fill[lightgray] (0,0) circle (2);
\fill[white](0.295081967213115*2,0) circle (0.591073979584261*2 and 0.923287071496965*2);

\draw (0.5*2,0.369735774108543*2)node[left]{$P$};
\draw (0.5*2,-0.369735774108543*2)node[left]{$Q$};

\coordinate (P) at (0.5*2,0.369735774108543*2);
\fill (P) circle [radius=1.5pt];
\coordinate (Q) at (0.5*2,-0.369735774108543*2);
\fill (Q) circle [radius=1.5pt];

 \end{tikzpicture}

    \caption{Ellipse associated with a point of rotation number $\frac{1}{3}$.}
    \label{fig:introduction2}
\end{figure}

\section{Description of Ellipse}
At the outset, we will verify the foundational knowledge required for understanding the rotation number in this paper.
For an orientation preserving homeomorphism $f$ on ${S^1}$, let $\rho_f$ be the rotation number of $f$; i.e.,
$$
\rho_f:=\lim_{n\to \infty}\dfrac{\overline{f}^n(x)-x}{n},
$$
where $x\in {\mathbb R}$ and $\overline{f}$ is the lift of $f$ on ${\mathbb R}$.

Then, following results (for example see \cite{KB}) are known.

\begin{lem}\label{fund-3}
Let $f$ be an orientation preserving homeomorphism  on ${S^1}$.
\begin{enumerate}
\item[(1)] $\rho_f$ is well defined.
\item[(2)] $\rho_f$ is a rational $\dfrac{p}{q}$, if and only if
there exits a $x\in {\mathbb R}$ such that
$\overline{f}^q(x)=x+p$, where $p$ and $q$ are relatively prime positive integers.  
\item[(3)] When $\rho_f$ is considered as a function on the set of  orientation preserving homeomorphisms on ${S^1}$, 
it is continuous on $C^0$ topology.
\end{enumerate}
\end{lem}

For a while, we will utilize the Poincaré hyperbolic disk model and the Poincaré half-plane model as hyperbolic geometries. Later on, we will employ the Beltrami-Klein disk model. In either model, let $d(P, Q)$ denote the hyperbolic distance between points $P$ and $Q$.

We define that $\delta(P,Q,R)$ is the altitude dropped on the $PQ$ from $R$.
$\delta(P,Q,R)$ is given by
\begin{align*}
\mathrm{arcsinh}\left(\dfrac{\sqrt{(-\cosh(\alpha-\gamma)+\cosh \beta)(\cosh(\alpha+\gamma)+\cosh \beta)}}{\sinh\gamma} \right),
\end{align*}
where $\alpha=d(Q,R), \beta=d(R,P)$ and $\gamma=d(P,Q)$.

For $P,Q\in D$, we define
\begin{align*}
 \Delta(P,Q):=\log \left(\dfrac{e^{d(P,Q)}+1}{e^{d(P,Q)}-1}\right)(=\log\ (\coth(\frac{d(P,Q)}{2}))). 
\end{align*}
The following theorem is a paraphrase of Theorem \ref{Dogru and Tabachnikov} authored by Dogru and Tabachnikov. An elementary proof is provided.

\begin{thm}\label{t1}
Let $\triangle PQR$ be a triangle in $D$.
Let $m\in {\mathbb Z}_{\geq 0}$ be the number of triangles inscribed in ${S^1}$ and 
circumscribing $\triangle PQR$. 
Then,
\begin{align*}
m=\begin{cases}0 & \text{if }\delta(P,Q,R)<\Delta(P,Q),\\
1 & \text{if }\delta(P,Q,R)=\Delta(P,Q),\\
2 & \text{if }\delta(P,Q,R)>\Delta(P,Q).
\end{cases}
\end{align*}
In particular, $\triangle PQR$ is large if and only if $\delta(P,Q,R)>\Delta(P,Q)$.
\end{thm}

\begin{proof}

Let $\mathbb{H}$ be the upper half plane $\{z \in {\mathbb C} \mid \Im(z)>0
%\text{the imaginary part of $z$ is positive}
\}$.
In the proof, we use the Poincaré half-plane model.
For $z_1,z_2\in \mathbb{H}$, we denote $d(z_1,z_2)$ by the distance from $z_1$ to $z_2$.
We see that %$SL(2,{\mathbb R})$ 
$PSL(2,{\mathbb R})$ is an isometry group of %$H$ 
$\mathbb{H}$ and it acts transitively on the set
$\{(z_1,z_2)\in \mathbb{H}^2\mid d(z_1,z_2)=c\}$, where $c$ can be any positive constant.
As a result,  %$g\in SL(2,{\mathbb R})$ 
$g\in PSL(2,{\mathbb R})$ exists such that $g(P)=i$ and $g(Q)$ is on the imaginary axis. 
Let $ki=g(Q)$ and we assume $k>1$ without loss of generality.
Since $P, Q$, and $R$ are not on any line, 
$g(R)$ is not on the imaginary axis. We may assume that the real part of $g(R)$ is positive without loss of generality.
Let $u+vi=g(R)$ with $u,v>0$.
In  hyperbolic geometry, a line in $\mathbb{H}$ is a semicircle that intersects the real axis perpendicularly.
We consider the circle $C_1$ with the center  $s$ and $-t, ki\in C_1$ for $s,t\in {\mathbb R}$ with $t>0$.
Then, we have
\begin{align*}
s^2+k^2=(s+t)^2,
\end{align*}
which implies $s=\frac{k^2-t^2}{2t}$.
Therefore, $2s+t=\frac{k^2}{t}$ is on $C_1$.
Similarly, the circle $C_2$ whose center is on the real axis and through  $-t$ and $i$ has the point 
$\frac{1}{t}$.
Let $C_3$ be the circle whose center is on the real axis and which passes through  $\frac{1}{t}$ and 
$\frac{k^2}{t}$. The following is the condition $u+vi\in C_3$:
\begin{align*}
\left(u-\dfrac{1+k^2}{2t}\right)^2+v^2=\left(\dfrac{k^2-1}{2t}\right)^2,
\end{align*}
which implies 
\begin{align}\label{(u^2+v^2)}
(u^2+v^2)t^2-(k^2+1)ut+k^2=0.
\end{align}
The discriminant $Di$ of formula (\ref{(u^2+v^2)}) for t is 
\begin{align*}
Di=(k^2+1)^2u^2-4k^2(u^2+v^2).
\end{align*}
We recall the definition of $m$ as the number of $C_3$, which includes $u+vi$. 
Since $Di=((k^2-1)u-2kv)((k^2-1)u+2kv)$ and $u,v>0$, we have
\begin{align}\label{m=begincases}
m=\begin{cases}0 & \text{if }v>\left(\frac{k^2-1}{2k}\right)u,\\
1 & \text{if }v=\left(\frac{k^2-1}{2k}\right)u,\\
2 & \text{if }v<\left(\frac{k^2-1}{2k}\right)u.
\end{cases}
\end{align}
We denote the set %$\{x+yi\in \mathbb{H}\mid y=\left(\dfrac{k^2-1}{2k}\right)x\}$
$\{x+yi\in \mathbb{H}\mid y=\frac{k^2-1}{2k}x\}$ by $L$.
Let $\theta\in (0,\pi/2)$ be the angle between the imaginary axis and $L$.
We will calculate the length $d$ of a vertical line from each point of L to the imaginary axis in the hyperbolic sense. Let $x+iy\in L$ with $r=\sqrt{x^2+y^2}$.
Then, the length of the curve $r\sin w +ir\cos w\ (0\leq w\leq \theta)$ is
\begin{align}\label{int_0}
d=\int_0^{\theta} \dfrac{dw}{\cos w}= \log\left(\dfrac{1+\tan(\theta/2)}{1-\tan(\theta/2)}\right).
\end{align}
We remark that $d$ depends only on $\theta$.
On the other hand, from the fact that $\tan \theta=\dfrac{2k}{k^2-1}$ we see that $\theta=\pi-2\arctan k$.
Therefore, from (\ref{int_0}),
we have
\begin{align}\label{d=log}
d=\log\left(\dfrac{k+1}{k-1}\right).
\end{align}
The distance between $i$ and $ki$ is
\begin{align}\label{int_1}
\int_1^{k} \dfrac{dy}{y}=\log k.
\end{align}
From (\ref{d=log}) and (\ref{int_1}), we have
\begin{align}\label{d=Delta}
d=\Delta(P,Q).
\end{align}
Similarly, we can consider the case of $u<0$.
The last statement of the theorem holds, as the condition of $\triangle PQR$ being large is equivalent to the existence of  $3$-orbits in the corresponding map on $S^1$.
\end{proof}

In the proof of Theorem \ref{t1}, we obtain the set containing the points $R$ with $m=1$.
As a result, we may express  Theorem \ref{t1} geometrically as follows.

\begin{thm}\label{t2}
Let $P, Q, R\in D$ be different from each other. 
Let $T_1, T_2\in {S^1}$ be two  points at infinity which intersect the hyperbolic line $PQ$.  
Let $C_1, C_2$ be two 
circular 
arcs in $D$ intersecting ${S^1}$ at $T_1,T_2$ with angle $2\arctan(e^{d(P,Q)})-\pi/2$.
Let $\Theta$ be the inner region bounded by $C_1, C_2, T_1$ and $T_2$.
The definition of $m\in {\mathbb Z}$ is the same as that of Theorem \ref{t1}. Then,
\begin{align*}
m=\begin{cases}0 & \text{if }R\in \Theta,\\
1 & \text{if } R\in C_1\cup C_2,\\
2 & \text{if } R\in (\mathrm{cl}\Theta)^{c}.
\end{cases}
\end{align*}
See Figure \ref{fig:Theorem2.3}.
\end{thm}

\begin{figure}
    \centering
\begin{tikzpicture}
\draw(0,0) circle (2);
\draw [](1,1.7320508) arc [start angle = 150, end angle = 210, radius = 1.7320508*2];
\draw (1,1.7320508)node[above]{$T_1$};
\draw (1,-1.7320508)node[below]{$T_2$};
\draw (2,0)node[right]{$T_3$};
\draw [](1,1.7320508) arc [start angle = 13.8978862480140, end angle = -13.8978862480140, radius = 7.21110255092798];
\draw [](1,1.7320508) arc [start angle = 106.102113751986, end angle = 253.897886248014, radius = 1.80277563773199];

\draw (0.280405516086274*2,0.207351901108928*2)node[left]{$P$};
\draw (0.280405516086274*2,-0.207351901108928*2)node[left]{$Q$};
\draw (1.9,0)node[left]{$C_1$};
\draw (-0.3,0)node[left]{$C_2$};

\coordinate (P) at (0.280405516086274*2,0.207351901108928*2);
\fill (P) circle [radius=1.5pt];
\coordinate (Q) at (0.280405516086274*2,-0.207351901108928*2);
\fill (Q) circle [radius=1.5pt];

 \end{tikzpicture}

    \caption{Two circular arcs $C_1$ and $C_2$.}
    \label{fig:Theorem2.3}
\end{figure}
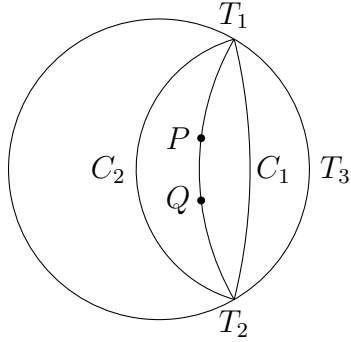

\begin{proof}
We recall the proof of Theorem \ref{t1}.
In the case of $u<0$, $L'$ is defined as 
$\{x+yi\in H\mid y=-\frac{k^2-1}{2k}x\}$.
We define $\Theta'$ as the inner region bounded by $L, L'$ and $0$, i.e.,
\begin{align*}
\Theta':=\{x+yi\in H\mid |y|>\left(\dfrac{k^2-1}{2k}\right)|x|\}.
\end{align*}
Then, we have
\begin{align}\label{g(f(R))}
m=\begin{cases}0 & \text{if }g(f(R))\in \Theta',\\
1 & \text{if } g(f(R))\in L\cup L',\\
2 & \text{if } g(f(R))\in (\mathrm{cl}\Theta')^{c}.
\end{cases}
\end{align}
At $0$ and $\infty$, the angle between $L$ and the imaginary axis 
is $2\arctan(e^{d(P,Q)})-\pi/2$.
The same is true for $L'$.
The theorem is derived from (\ref{g(f(R))}) and the preceding facts. 
\end{proof}
The arcs will  be  discussed  in Proposition \ref{p2}, as they are required  in the following chapter. 

\begin{prop}\label{p2}
Let $P, Q\in D$ be different from each other. 
Let $T_1, T_2\in {S^1}(=\mathbb{R}/\mathbb{Z})$ be two points at infinity, which intersect the hyperbolic line $PQ$ with
$T_1-T_2\in (0,1/2]\  mod\  1$.
Let $T_3\in S^1(={\mathbb R}/{\mathbb Z})$ be the point such that $T_1+T_2=2T_3\  mod\  1$ and $T_3-T_2\in (0,1/2]\  mod\  1$.
Let $C_1, C_2$ be two 
circular 
arcs in $D$ intersecting ${S^1}$ at $T_1,T_2$ with angle $2\arctan(e^{d(P,Q)})-\pi/2$ such that
$C_1$ is closer to $T_3$ than $C_2$. Then, the center of $C_1$ is 
\begin{align*}
\dfrac{e^{2d(P,Q)}-1}{u(e^{2d(P,Q)}-1)-2e^{d(P,Q)}\sqrt{1-u^2}}T_3,
\end{align*}
and the center of $C_2$ is
\begin{align*}
\dfrac{e^{2d(P,Q)}-1}{u(e^{2d(P,Q)}-1)+2e^{d(P,Q)}\sqrt{1-u^2}}T_3,
\end{align*}
where $u$ is the length between the origin and the straight line $T_1T_2$ in the sense of the Euclidean norm.
\end{prop}
\begin{proof}
We put $k=e^{d(P,Q)}$.Let $A=aT_3$ be the center of $C_1$ where $a\in {\mathbb R}$.
See Figure \ref{fig:prpposition2.4}.
Let $\theta$ be $\arccos u$.
We assume that $a>0$.
Let $B$ be the intersection of line $OT_3$ and the line  vertical to  line $AT_1$
that passes through $T_1$.
Let $C$ be the intersection of line $OT_3$ and the line vertical to the line $OT_1$ that passes through $T_1$.
Since $\angle BT_1C=2\arctan(k)-\pi/2$ and $\angle T_1OA=\theta$, 
we have $\angle OAT_1=2\arctan k-\pi/2-\theta$.
Therefore, $AT_1=\sqrt{1-u^2}/\sin(2\arctan k-\pi/2-\theta)$,
which implies
\begin{align*}
a=u+\dfrac{\sqrt{1-u^2}}{\tan(2\arctan k-\pi/2-\theta)}\\
=\dfrac{k^2-1}{u(k^2-1)-2k\sqrt{1-u^2}}.
\end{align*}
For the case of $a<0$ we have the same formula.
We can prove the formula for $C_2$ in the same manner.
\begin{figure}
    \centering
\begin{tikzpicture}
\draw[->] (-0.5*1.5, 0)--(4*1.5,0);
\draw[->] (0, -0.5*1.5)--(0,2);
\draw (0, 0)--(1*1.5,1*1.5);
\draw (1*1.5,1*1.5)--(3*1.5,0);
\draw (0, 0)--(1*1.5,1*1.5);
\draw (1*1.5,1*1.5)--(1,0);
\draw (1*1.5,1*1.5)--(2*1.5,0);
\draw (1*1.5,1*1.5)--(1.5,0);

\draw [dashed](1*1.5,1*1.5)--(1*1.5+0.8*0.5,1*1.5+0.8*1.5);
\draw [dashed](1*1.5,1*1.5)--(1*1.5-0.8,1*1.5+0.8);

\draw (0,-0.3)node[left]{$O$};
\draw (1*1.5-0.1,1*1.5)node[above]{$T_1$};
\draw (1,0)node[below]{$B$};
\draw (2*1.5,0)node[below]{$C$};
\draw (3*1.5,0)node[below]{$A$};
\draw (1.7,0)node[below]{$uT_3$};

 \end{tikzpicture}

    \caption{Points $A$, $B$, $C$.}
    \label{fig:prpposition2.4}
\end{figure}
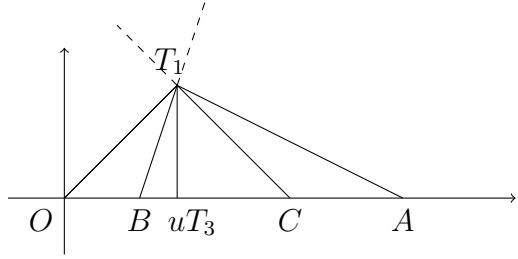
\end{proof}
Now, we use  the Beltrami-Klein disk model (for example, see \cite{B},\cite{L}).
Let $D$ be the  disk $\{(x,y)\in {\mathbb R}^2\mid x^2+y^2<1\}$.
The boundary $\partial D (= S^1)$ is regarded as points at infinity.
Then, the distance $d(P,Q)$ between $P$ and $Q$  is given by
\begin{align*}
\dfrac{1}{2}\left|\log\left(\dfrac{|v_1Q||v_2P|}{|v_1P||v_2Q|} \right)\right|,
\end{align*}
where $v_1, v_2\in {S^1}$ be two points at infinity which intersect the hyperbolic line $PQ$ (see Figure \ref{distance}).

\begin{figure}
    \centering
\begin{tikzpicture}
\draw(0,0) circle (2);
\draw (-1/2,1.732/2)--(1,0);
 \draw (-1/2,1.732/2)node[left]{$P$};
  \draw (1,0+0.1)node[right]{$Q$};

\draw[dashed] (-1/2,1.732/2)--(-1.42705098312484, 1.40125853844407);
\draw[dashed] (1,0)--(1.92705098312484,-0.535233134659635);
\draw (-1.42705098312484, 1.40125853844407)node[left]{$v_1$};
\draw (1.92705098312484,-0.535233134659635)node[right]{$v_2$};
\coordinate (P) at (-1/2,1.732/2);
\fill (P) circle [radius=1.5pt];
\coordinate (Q) at (1,0);
\fill (Q) circle [radius=1.5pt];

 \end{tikzpicture}

    \caption{$d(P,Q)$.}
    \label{distance}
\end{figure}

We define  $G:D\to D$ by for $(x,y)\in D$
\begin{align*}
G(x,y):=\left(\dfrac{2x}{1+x^2+y^2}, \dfrac{2y}{1+x^2+y^2} \right).
\end{align*}
Then, the inverse map of $G$ is given by for $(x,y)\in D$
\begin{align}\label{G-1}
G^{-1}(x,y)=\left(\dfrac{x}{1+\sqrt{1-x^2-y^2}},\dfrac{y}{1+\sqrt{1-x^2-y^2}}\right).
\end{align}
$G$ is naturally extended to the boundary of $D$. 
The Poincaré hyperbolic disk model
is isomorphic to the Beltrami–Klein model via the map $G$.
A line of $D$ in the Beltrami–Klein model is known to be a straight line segment with endpoints at the boundary of $D$.

We give an example.
\begin{example}\label{example1}
Let $P=(-\frac{1}{4}, \frac{\sqrt{3}}{4})$, $Q=(-\frac{1}{4}, -\frac{\sqrt{3}}{4})$,
$R=(\frac{1}{2}, 0)$, and $S=(-\frac{1}{4}, 0)$.
We set $v_1=(-\frac{1}{4}, \frac{\sqrt{15}}{4})$, 
$v_2=(-\frac{1}{4}, -\frac{\sqrt{15}}{4})$,
$v_3=(1,0)$,
and 
$v_4=(-1,0)$, which are points at infinity.
Then, $\triangle PQR$ is an equilateral triangle in $D$. See Figure
\ref{fig:example3-1}.
Then, 
\begin{align*}
&d(P,Q)=\dfrac{1}{2}\log\dfrac{|v_1Q||Pv_2|}{|v_2Q||Pv_1|}
=\log \dfrac{\sqrt{5}+1}{\sqrt{5}-1},\\
&\Delta(P,Q)=\log\dfrac{e^{d(P,Q)}+1}{e^{d(P,Q)}-1}=\log \sqrt{5}.
\end{align*}
Conversely, we have
\begin{align*}
\delta(P,Q,R)=d(R,S)=\dfrac{1}{2}\log\dfrac{|v_3S||Rv_4|}{|v_4S||Rv_3|}=\log \sqrt{5}.
\end{align*}
Therefore, $m$ (defined in Theorem \ref{t1}) is $1$.
Therefore, $\triangle PQR$ is not large.

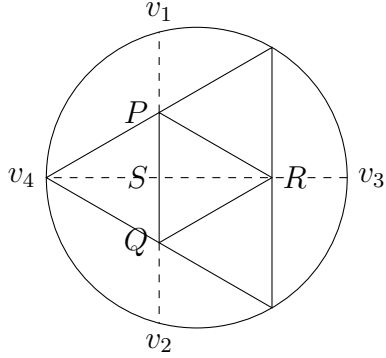
\begin{figure}
    \centering
\begin{tikzpicture}
\draw(0,0) circle (2);
\draw (-1/2,1.732/2)--(-1/2,-1.732/2)--(1,0)--cycle;
 \draw (-1/2,1.732/2)node[left]{$P$};
 \draw (-1/2,-1.732/2)node[left]{$Q$};
  \draw (1,0)node[right]{$R$};
 \draw (-1/2,0)node[left]{$S$};
 
\draw[dashed] (-1/2,1.732/2)--(-1/2,1.936491673103708);
\draw[dashed] (-1/2,-1.732/2)--(-1/2,-1.936491673103708);
\draw[dashed] (2, 0)--(-2,0);
 \draw (-1/2,1.936491673103708)node[above]{$v_1$};
\draw (-1/2,-1.936491673103708)node[below]{$v_2$};
\draw (-2, 0)node[left]{$v_4$};
\draw (2,0)node[right]{$v_3$};
\draw (1,1.732)--(1,-1.732)--(-2,0)--cycle;

 \end{tikzpicture}
    \caption{$\triangle PQR$ in Example\ref{example1}.}
    \label{fig:example3-1}
\end{figure}
\end{example}

\begin{thm}
\label{t4}
Let $\triangle PQR$ be a triangle in $D$ (Beltrami–Klein model).
Let $T_1, T_2\in {S^1}(=\mathbb{R}/\mathbb{Z})$ be two points at infinity, which intersect the hyperbolic line $PQ$ with
$T_1-T_2\in (0,1/2]\  mod\  1$.
Let $T_3, T_4\in S^1(={\mathbb R}/{\mathbb Z})$ be the points such that $T_1+T_2=2T_3\  mod\  1$, $T_3-T_2\in (0,1/2]\  mod\  1$, and $T_4=T_3+\frac{1}{4}\  mod\  1$.
Let $(x_1,x_2)'$ be the coordinates using the base $\{T_3, T_4\}$.
Let $u$ be the $x_1$ coordinate of $P$.
We define $a\in {\mathbb R}$ as
\begin{align*}
\dfrac{e^{2d(P,Q)}-1}{u(e^{2d(P,Q)}-1)-2e^{d(P,Q)}\sqrt{1-u^2}}.
\end{align*}
Let $E$ be the ellipse defined by
\begin{align*}
\{(x_1,x_2)'\mid \dfrac{(x_1-C)^2}{A^2}+\dfrac{x_2^2}{B^2}=1\},
\end{align*}
where
\begin{align*}
&A:=\dfrac{|au-1|\sqrt{a^2-2ua+1}}{(u^2+1)a^2-2au+1}\left(=\dfrac{2k(k^2+1)(1-u^2)}{(k^2-2ku+1)(k^2+2ku+1)}\right),\\
&B:=\dfrac{\sqrt{a^2-2ua+1}}{\sqrt{(u^2+1)a^2-2au+1}}\left(=\dfrac{(k^2+1)\sqrt{1-u^2}}{\sqrt{(k^2-2ku+1)(k^2+2ku+1)}} \right),\\
&C:=\dfrac{a^2u}{(u^2+1)a^2-2au+1}\left(=\dfrac{(k^2-1)^2u}{(k^2-2ku+1)(k^2+2ku+1)} \right),\\
\end{align*}
where $k:=e^{d(P,Q)}$. See Figure \ref{fig:Theorem3.2}.
Then, $E$ is included in $D\cup {S^1}$ and 
is tangent to ${S^1}$ at $T_1$ and $T_2$.
Let $\Phi$ be the inner region bounded by $E$.
The definition of $m\in {\mathbb Z}$ is the same as that of Theorem \ref{t1}. Then,
\begin{align*}
m=\begin{cases}0 & \text{if }R\in \Phi,\\
1 & \text{if } R\in E,\\
2 & \text{if } R\in (\mathrm{cl}\Phi)^{c}.
\end{cases}
\end{align*}
\end{thm}
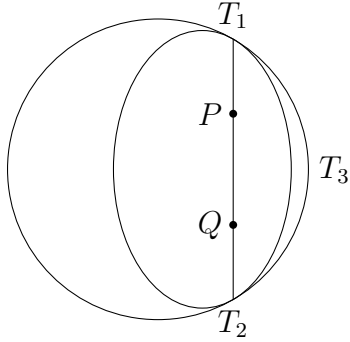
\begin{figure}
    \centering
    \begin{tikzpicture}
\draw(0,0) circle (2);
\draw (1,1.7320508)node[above]{$T_1$};
\draw (1,-1.7320508)node[below]{$T_2$};
\draw (2,0)node[right]{$T_3$};

\draw (0.5*2,0.369735774108543*2)node[left]{$P$};
\draw (0.5*2,-0.369735774108543*2)node[left]{$Q$};
\draw(0.295081967213115*2,0) circle (0.591073979584261*2 and 0.923287071496965*2);
\draw (1,1.7320508)--(1,-1.7320508);

\coordinate (P) at (0.5*2,0.369735774108543*2);
\fill (P) circle [radius=1.5pt];
\coordinate (Q) at (0.5*2,-0.369735774108543*2);
\fill (Q) circle [radius=1.5pt];

 \end{tikzpicture}

    \caption{Ellipse $E$.}
    \label{fig:Theorem3.2}
\end{figure}

\begin{proof}
For simplicity, we assume that $P=(u,p), Q=(u,q)$, where 
$0\leq u<1,\ p>q$.
First, we show that a circle with a center on the $x$-axis maps to an ellipse with a center on the $x$-axis by the map $G$.
We define circle $C(r,b)$ by $\{(x,y)\mid (x-b)^2+y^2=r^2\}$.
The map $G$ is extended to the map $\Bar{G}$ from ${\mathbb R}^2$ to $D\cup {S^1}$ using the same formula 
(\ref{G-1}).
Then, using simple calculations, $\Bar{G}$ bijectively maps  $C(r,b)$  to the ellipse $E(r,b)$ defined by
\begin{align*}
\dfrac{(x-C')^2}{A'^2}+\dfrac{y^2}{B'^2}=1,
\end{align*}
where 
\begin{align*}
&A':=\dfrac{2r(1+r^2-b^2)}{((r-b)^2+1)((r+b)^2+1)},\\
&B':=\dfrac{2r}{\sqrt{(r-b)^2+1}\sqrt{(r+b)^2+1}},\\
&C':=\dfrac{2b(1-r^2+b^2)}{((r-b)^2+1)((r+b)^2+1)}.
\end{align*}
For $V_1,V_2\in C(r,b)$, we denote $\mathrm{arc}(V_1V_2)$
by the arc of $C(r,b)$ that goes counterclockwise from $V_1$ to $V_2$. 
For $V_1,V_2\in E(r,b)$, we denote $\mathrm{arc}'(V_1V_2)$
by the arc of $E(r,b)$ that goes counterclockwise from $V_1$ to $V_2$.
We assume that $C(r,b)$ intersects ${S^1}$ at $T_1=(u,\sqrt{1-u^2}), T_2=(u,-\sqrt{1-u^2})$. 
Then, we have $r^2=1+b^2-2bu$, $C(r,0)=S^1$, and $C(r,\frac{1}{u})$ is a line in the hyperbolic sense.
We assume $%n
b \ne 0,\frac{1}{u}$. Then, we have
\begin{enumerate}
    \item[1)]  
    $\mathrm{arc}(T_2T_1)$ is mapped to $\mathrm{arc}'(T_2T_1)$ by $\Bar{G}$ if $b>\dfrac{1}{u}$ or $b<0$, and  $\mathrm{arc}(T_1T_2)$ is mapped to $\mathrm{arc}'(T_1T_2)$ by $\Bar{G}$ if $0<b<\dfrac{1}{u}$,
     where for the case of $u=0$ we set $\dfrac{1}{u}:=\infty$.
    \item[2)]
    Furthermore, we assume that $C(r',b')$ intersects ${S^1}$ at $T_1, T_2$, and $b'=\dfrac{b}{2bu-1}$.
    Then, $E(r,b)=E(r',b')$.
\end{enumerate}
The proofs of 1) and 2) are left to the reader.
Theorem \ref{t2}, Proposition \ref{p2}, 1), and 2) can be used to prove the theorem. 
In Proposition \ref{p2}, we set $G^{-1}(P)$ (resp., $G^{-1}(Q))$ as $P$ (resp., $Q$).
$C_1$ in Proposition \ref{p2}, is the arc of $C(\sqrt{(a-u)^2+1-u^2},a)$, and 
$C_2$ is the arc of $C(\sqrt{(a'-u)^2+1-u^2},a')$, where
$a':=a/(2au-1)$.
As a result of 1) and 2), we have completed the proof.
\end{proof}

In Theorem \ref{t4}
let $P=(u,p)'$ and $Q=(u,q)'$.
we have
\begin{align*}
d(P,Q)=\dfrac{1}{2}\left|\log\left(\dfrac{(p+\sqrt{1-u^2})(-q+\sqrt{1-u^2})}{(q+\sqrt{1-u^2})(-p+\sqrt{1-u^2})}\right)\dfrac{}{}\right|.
\end{align*}
Then, the ellipse $E$ in Theorem \ref{t4} is given by
\begin{prop}\label{p3}
\begin{align*}
\{(x_1,x_2)'\mid \dfrac{(x_1-C)^2}{A^2}+\dfrac{x_2^2}{B^2}=1\},
\end{align*}
where
\begin{align*}
&A:=\sqrt{\dfrac{(p^2+u^2-1)(q^2+u^2-1)(pq+u^2-1)^2}{(p^2(q^2+u^2)-2pq+(q^2-2)u^2+u^4+1)^2}},\\
&B:=\sqrt{\dfrac{(pq+u^2-1)^2}{p^2(q^2+u^2)-2pq+(q^2-2)u^2+u^4+1}},\\
&C:=\dfrac{u(p-q)^2}{p^2(q^2+u^2)-2pq+(q^2-2)u^2+u^4+1}.
\end{align*}
\end{prop}

\begin{rem}
Proposition \ref{p3} can also be demonstrated without the use of hyperbolic geometry.
Given two points $P=(u,p), Q=(u,q)$ in $D$,
let us characterize the region $\mathcal{R}_{P,Q} \subset D$ such that for all $R\in \mathcal{R}_{P,Q}$ there exists a triangle $\triangle u_1u_2u_3$ inscribed in ${S^1}$ and circumscribing $\triangle PQR$.
%for fixed $P=(u,p), Q=(u,q)$ in $D$.
If $\triangle u_1u_2u_3$ exists, we can infer that $u_1=(\cos\theta, \sin\theta)$ and $\overline{u_1u_2}, \overline{u_1u_3}$ include $P, Q$.
Then, $R$ lies on the open segment $\sigma_\theta:=\overline{u_2u_3}$, and, for any point $R' \in \sigma_\theta$, $\triangle u_1u_2u_3$ circumscribes $\triangle PQR'$.
Thus, $\sigma_\theta \subset \mathcal{R}_{P,Q}$ and therefore, $\mathcal{R}_{P,Q}=\bigcup_{\theta\in [0,2\pi]}\sigma_\theta$.
As $\mathcal{R}_{P,Q}$ is a union of segments, its boundary component in $D$ is the envelope of the family of those segments (see Figure \ref{fig:Remark3.1}).

\begin{figure}
    \centering
\begin{tikzpicture}[x= 1cm ,y= 1cm ]%
\clip (-2.5,-2.5) rectangle (2.75,2.75);%
\draw [line width=0.6](2.00000,0.00000)--(1.98423,0.25067)--(1.93717,0.49738)--(1.85955,0.73625)--(1.75261,0.96351)--(1.61803,1.17557)--(1.45794,1.36909)--(1.27485,1.54103)--(1.07165,1.68866)--(0.85156,1.80965)--(0.61803,1.90211)--(0.37476,1.96457)--(0.12558,1.99605)--(-0.12558,1.99605)--(-0.37476,1.96457)--(-0.61803,1.90211)--(-0.85156,1.80965)--(-1.07165,1.68866)--(-1.27485,1.54103)--(-1.45794,1.36909)--(-1.61803,1.17557)--(-1.75261,0.96351)--(-1.85955,0.73625)--(-1.93717,0.49738)--(-1.98423,0.25067)--(-2.00000,-0.00000)--(-1.98423,-0.25067)--(-1.93717,-0.49738)--(-1.85955,-0.73625)--(-1.75261,-0.96351)--(-1.61803,-1.17557)--(-1.45794,-1.36909)--(-1.27485,-1.54103)--(-1.07165,-1.68866)--(-0.85156,-1.80965)--(-0.61803,-1.90211)--(-0.37476,-1.96457)--(-0.12558,-1.99605)--(0.12558,-1.99605)--(0.37476,-1.96457)--(0.61803,-1.90211)--(0.85156,-1.80965)--(1.07165,-1.68866)--(1.27485,-1.54103)--(1.45794,-1.36909)--(1.61803,-1.17557)--(1.75261,-0.96351)--(1.85955,-0.73625)--(1.93717,-0.49738)--(1.98423,-0.25067)--(2.00000,-0.00000);%
\draw (0.5000000,0.4200000)node[above]{$P$};
\draw (0.5000000,-1.430000)node[below]{$Q$};
\draw (-1.400000,-0.15000000)node[below]{$R$};
\draw (1.8800000,0.4000000)node[right]{$u_1$};
\draw (-1.8400000,0.7200000)node[left]{$u_2$};
\draw (0.1200000,-1.950000)node[below]{$u_3$};

\draw [line width=0.3](1.96013,0.39734)--(-1.88525,0.66771)--(0.11789,-1.99652)--(1.96013,0.39734);%
\draw [line width=0.6](0.50000,0.50000)--(0.50000,-1.50000)--(-1.21754,-0.22037)--(0.50000,0.50000);%
\draw [line width=0.006](-1.60000,1.20000)--(0.00000,-2.00000);%
\draw [line width=0.006](-1.74231,0.98202)--(0.05649,-1.99920);%
\draw [line width=0.006](-1.86204,0.72993)--(0.10707,-1.99713);%
\draw [line width=0.006](-1.94987,0.44497)--(0.15276,-1.99416);%
\draw [line width=0.006](-1.99568,0.13132)--(0.19436,-1.99053);%
\draw [line width=0.006](-1.98964,-0.20327)--(0.23253,-1.98644);%
\draw [line width=0.006](-1.92368,-0.54723)--(0.26778,-1.98199);%
\draw [line width=0.006](-1.79321,-0.88567)--(0.30056,-1.97729);%
\draw [line width=0.006](-1.59865,-1.20179)--(0.33121,-1.97238);%
\draw [line width=0.006](-1.34624,-1.47907)--(0.36005,-1.96732);%
\draw [line width=0.006](-1.04759,-1.70369)--(0.38732,-1.96214);%
\draw [line width=0.006](-0.71816,-1.86661)--(0.41325,-1.95684);%
\draw [line width=0.006](-0.37494,-1.96454)--(0.43802,-1.95145);%
\draw [line width=0.006](-0.03399,-1.99971)--(0.46180,-1.94596);%
\draw [line width=0.006](0.29134,-1.97867)--(0.48472,-1.94037);%
\draw [line width=0.006](0.59137,-1.91057)--(0.50692,-1.93469);%
\draw [line width=0.006](0.86022,-1.80555)--(0.52852,-1.92890);%
\draw [line width=0.006](1.09526,-1.67344)--(0.54961,-1.92300);%
\draw [line width=0.006](1.29640,-1.52294)--(0.57030,-1.91697);%
\draw [line width=0.006](1.46521,-1.36130)--(0.59066,-1.91079);%
\draw [line width=0.006](1.60432,-1.19422)--(0.61079,-1.90445);%
\draw [line width=0.006](1.71680,-1.02597)--(0.63077,-1.89793);%
\draw [line width=0.006](1.80583,-0.85965)--(0.65067,-1.89120);%
\draw [line width=0.006](1.87448,-0.69737)--(0.67057,-1.88423);%
\draw [line width=0.006](1.92557,-0.54052)--(0.69055,-1.87700);%
\draw [line width=0.006](1.96162,-0.38992)--(0.71068,-1.86947);%
\draw [line width=0.006](1.98482,-0.24598)--(0.73105,-1.86160);%
\draw [line width=0.006](1.99704,-0.10883)--(0.75174,-1.85335);%
\draw [line width=0.006](1.99988,0.02160)--(0.77283,-1.84465);%
\draw [line width=0.006](1.99470,0.14549)--(0.79443,-1.83545);%
\draw [line width=0.006](1.98261,0.26313)--(0.81663,-1.82568);%
\draw [line width=0.006](1.96456,0.37484)--(0.83955,-1.81526);%
\draw [line width=0.006](1.94131,0.48096)--(0.86331,-1.80408);%
\draw [line width=0.006](1.91349,0.58184)--(0.88806,-1.79202);%
\draw [line width=0.006](1.88164,0.67782)--(0.91395,-1.77896);%
\draw [line width=0.006](1.84615,0.76923)--(0.94118,-1.76471);%
\draw [line width=0.006](1.80738,0.85638)--(0.96994,-1.74906);%
\draw [line width=0.006](1.76557,0.93955)--(1.00048,-1.73178);%
\draw [line width=0.006](1.72094,1.01900)--(1.03309,-1.71252);%
\draw [line width=0.006](1.67362,1.09498)--(1.06810,-1.69091);%
\draw [line width=0.006](1.62373,1.16770)--(1.10592,-1.66642);%
\draw [line width=0.006](1.57131,1.23734)--(1.14702,-1.63840);%
\draw [line width=0.006](1.51639,1.30406)--(1.19197,-1.60599);%
\draw [line width=0.006](1.45896,1.36801)--(1.24146,-1.56805);%
\draw [line width=0.006](1.39898,1.42929)--(1.29631,-1.52302);%
\draw [line width=0.006](1.33637,1.48799)--(1.35749,-1.46875);%
\draw [line width=0.006](1.27105,1.54416)--(1.42612,-1.40220);%
\draw [line width=0.006](1.20289,1.59783)--(1.50344,-1.31897);%
\draw [line width=0.006](1.13173,1.64899)--(1.59056,-1.21248);%
\draw [line width=0.006](1.05741,1.69761)--(1.68797,-1.07273);%
\draw [line width=0.006](0.97972,1.74360)--(1.79402,-0.88403);%
\draw [line width=0.006](0.89844,1.78684)--(1.90100,-0.62144);%
\draw [line width=0.006](0.81332,1.82716)--(1.98479,-0.24620);%
\draw [line width=0.006](0.72408,1.86433)--(1.97823,0.29428);%
\draw [line width=0.006](0.63042,1.89804)--(1.71504,1.02890);%
\draw [line width=0.006](0.53203,1.92794)--(0.90137,1.78537);%
\draw [line width=0.006](0.42858,1.95354)--(-0.50811,1.93438);%
\draw [line width=0.006](0.31971,1.97428)--(-1.68044,1.08448);%
\draw [line width=0.006](0.20507,1.98946)--(-1.99956,-0.04186);%
\draw [line width=0.006](0.08434,1.99822)--(-1.81005,-0.85071);%
\draw [line width=0.006](-0.04279,1.99954)--(-1.49361,-1.33009);%
\draw [line width=0.006](-0.17656,1.99219)--(-1.19399,-1.60449);%
\draw [line width=0.006](-0.31710,1.97470)--(-0.94253,-1.76398);%
\draw [line width=0.006](-0.46441,1.94533)--(-0.73744,-1.85908);%
\draw [line width=0.006](-0.61821,1.90206)--(-0.57014,-1.91701);%
\draw [line width=0.006](-0.77790,1.84252)--(-0.43228,-1.95272);%
\draw [line width=0.006](-0.94241,1.76405)--(-0.31718,-1.97469);%
\draw [line width=0.006](-1.11002,1.66369)--(-0.21976,-1.98789);%
\draw [line width=0.006](-1.27814,1.53829)--(-0.13621,-1.99536);%
\draw [line width=0.006](-1.44313,1.38470)--(-0.06368,-1.99899);%
\draw [line width=0.006](-1.60000,1.20000)--(0.00000,-2.00000);%
\end{tikzpicture}%
    \caption{Ellipse as envelope.}
    \label{fig:Remark3.1}
\end{figure}
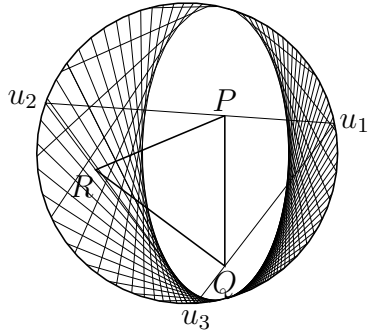

Let us now directly deduce the envelope equation.
The segment $\sigma_\theta$ is contained within the line described by the following equation:
\begin{align*}
    F(x,y,\theta)&:=
    \left( (u^2+1-p q) \cos\theta+(p+q) u \sin\theta-2 u\right)x\\
   &\hspace{7mm}+ \left((p+q) u \cos\theta+(1+p q-u^2)\sin\theta-p-q\right)y\\
   &\hspace{7mm}-2 u \cos\theta-(p+q) \sin\theta+
   p q+u^2+1\\
   &=0
\end{align*}
By eliminating the variable $\theta$ from $F(x,y,\theta)=0$ and $\frac{\partial }{\partial\theta}F(x,y,\theta)=0$, we obtain the same equation of ellipse as in Proposition \ref{p3}.
\end{rem}

\section{Large triangle}
In this section, for large triangles, their sizes are measured from the perspective of standard Euclidean geometry, and their characteristics are investigated.
In this chapter, basic properties of the rotation number are employed. From Lemma \ref{fund-3}(3), it is worth noting that as the triangle $\triangle PQR$ is continuously moved, the rotation number $\rho(\triangle PQR)$ also undergoes continuous changes.
The following lemma is also useful.

\begin{lem}[Lemma 1 \cite{DT}]
 Let $C_1\subset C_2$ be two convex polygons in $D$. Then $\rho(C_1) \geq \rho(C_2)$.   
\end{lem}

\begin{thm}\label{t7}
Let $\triangle PQR$ be a triangle in $D$.
Then, there exists a triangle $\triangle P'Q'R'$  in $D$
such that $\triangle P'Q'R'$ is  similar to $\triangle PQR$ and
for arbitrary triangle $\triangle ABC\subset D$ which is congruent to $\triangle P'Q'R'$, 
$\triangle ABC$ is large.
Here, similarity and congruence are in the sense of Euclidean geometry.
\end{thm}

\begin{proof}
First, we assume that
$\triangle PQR$ is an obtuse triangle or right triangle.
We may assume $\angle R\geq \pi/2$.
We set $P_1=(0,1)$ and $P_2=(0,-1)$.
It is not difficult to find $P_3=(x_1,x_2)\in D\cup {S^1}$
such that $\triangle P_1P_2P_3$ is similar to $\triangle PQR$.
We assume $x_1>0$ without loss of generality.
We take a look at similarity transformation $\mathcal{S}_{\lambda}$ for $\lambda\in {\mathbb R}$ defined as
for $X\in {\mathbb R}^2$,
\begin{align*}
\mathcal{S}_{\lambda}(X):=\lambda X.
\end{align*}
We take $K>0$ such that for any $k\geq K$
\begin{align}\label{frac4k}
\frac{4k(k^2+1)}{(k-1)^4}<\frac{x_1}{4}.
\end{align}
We take
\begin{align}\label{epsilon=min}
\epsilon=\min \{\frac{x_1^2}{32}, \frac{1}{4K}\}.
\end{align}
Let $\lambda=1-\epsilon$.
Let $P_i'=\mathcal{S}_{\lambda}(P_i')$ for $i=1,2,3$.
Let us demonstrate that $\triangle P_1'P_2'P_3'$ is as specified in the theorem.
We note that when considering a congruent transformation to $\triangle P_1'P_2'P_3'$ in $D$, we just need to consider translations, not rotations or %symmetry transformations 
reflections, 
by symmetry.
For $\tau_1,\tau_2\in {\mathbb R}^2$, we define the translation %$\mathcal{P}_{\tau_1,\tau_2}$ as
$\mathcal{P}_{\tau_1,\tau_2}(x,y):=(x+\tau_1,y+\tau_2)$
for $(x,y)\in {\mathbb R}^2$.
We define the set $U\subset D$ as
\begin{align*}
U:=\{(\tau_1,\tau_2)\in {\mathbb R}^2\mid \mathcal{P}_{\tau_1,\tau_2}(\triangle P_1'P_2'P_3')\subset D\}.
\end{align*}
Then, we see
\begin{align}\label{Usubset}
U\subset \{(\tau_1,\tau_2)\in {\mathbb R}^2\mid \mathcal{P}_{\tau_1,\tau_2}(P_1'P_2')\subset D\}
\subset (-\sqrt{2\epsilon},\sqrt{2\epsilon})\times (-\epsilon,\epsilon).
\end{align}
Let $(\tau_1',\tau_2')\in U$.
We set $V_i:=\mathcal{P}_{\tau_1',\tau_2'}(P_i)$ for $i=1,2,3$.
From the facts that $|\tau_2'|<\epsilon$, $(\tau_1',\tau_2')\in U$, and (\ref{epsilon=min}), we get
\begin{align*}
&d(V_1,V_2)\\
&=\dfrac{1}{2}\log\left(\dfrac{(1-\epsilon+\tau_2'+\sqrt{1-(\tau_1')^2})(1-\epsilon-\tau_2'+\sqrt{1-(\tau_1')^2})}{(-1+\epsilon-\tau_2'+\sqrt{1-(\tau_1')^2})(-1+\epsilon+\tau_2'+\sqrt{1-(\tau_1')^2})}\right)\\
&> 
\dfrac{1}{2}\log\dfrac{(1-2\epsilon)^2}{(2\epsilon)^2}>\log\dfrac{1}{4\epsilon}
.
\end{align*}
We set $k':=e^{d(V_1,V_2)}$.
Then, we have
\begin{align}\label{k'>}
k'>\dfrac{1}{4\epsilon}.
\end{align}
Let $E$ be the ellipse defined for $V_1,V_2$ in Theorem \ref{t4}; i.e., 
if $A,B$, and $C$ are as defined in Theorem  \ref{t4}, we have
\begin{align*}
E=\{(x,y)\mid \dfrac{(x-C)^2}{A^2}+\dfrac{y^2}{B^2}=1\}.
\end{align*}
In particular, $A$ is as follows:
\begin{align}\label{A:=}
A:=\dfrac{2k'(k'^2+1)(1-\tau_1'^2)}{(k'^2-2k'\tau_1'+1)(k'^2+2k'\tau_1'+1)}.
\end{align}
 Consider $\Phi$ to be the inner region enclosed by $E$.
Since $(\tau_1',0)$ is in $\Phi$, 
$\Phi$ is included in
the set $\{(x,y)\mid |x-\tau_1'|<2A\}$ denoted by $\Phi'$.
Let us demonstrate $V_3\notin\Phi'$.
From (\ref{frac4k}), (\ref{epsilon=min}), (\ref{k'>}), and (\ref{A:=}) we have
\begin{align}\label{2A<frac}
2A<\frac{4k'(k'^2+1)}{(k'-1)^4}<\frac{x_1}{4}.
\end{align}
%From (\ref{epsilon=min}) and (\ref{Usubset}), we have 
%\begin{align}\label{tau_1}
%|\tau_1'|<\sqrt{2\epsilon}<\frac{x_1}{4}.
%\end{align}
We recall $V_3=\mathcal{P}_{\tau_1',\tau_2'}(P_3')=((1-\epsilon)x_1+\tau_1',(1-\epsilon)x_2+\tau_2')$.
From (\ref{epsilon=min}) and (\ref{2A<frac}), we have
\begin{align*}
(1-\epsilon)x_1>\dfrac{31}{32}x_1>2A,
\end{align*}
which implies $V_3\notin\Phi'$.
Therefore, from Theorem \ref{t4}, we observe that 
$\triangle V_1V_2V_3$ is large.
Next, we assume that
$\triangle PQR$ is an acute triangle.
Then, there is $\triangle P_1P_2P_3$, which
is inscribed in ${S^1}$ and is similar to $\triangle PQR$.
We can assume that $P_1P_2$ is parallel to the $y$-axis.
We can provide comparable proof in this case as  we did in the previous ones.
We leave the proof to the reader.
\end{proof}
If $\triangle PQR$ is similar to $\triangle P'Q'R'$,
we denote by $\triangle PQR\sim \triangle P'Q'R'$.
For a triangle $\triangle PQR$ in ${\mathbb R}^2$, we define $\kappa(\triangle PQR)$ as
its circumradius if $\triangle PQR$  is an acute triangle, and
$\kappa(\triangle PQR)$ as half of the length of the largest side of $\triangle PQR$
if $\triangle PQR$  is an obtuse triangle or right triangle.

For a triangle $\triangle PQR$ in ${\mathbb R}^2$, we define $\mu(\triangle PQR)$ as 
the infimum of $\kappa(\triangle P'Q'R')$ such that
\begin{enumerate}
\item[(1)] $\triangle P'Q'R'\subset D$,
\item[(2)] $\triangle P'Q'R'\sim \triangle PQR$,
\item[(3)] for any $\triangle P''Q''R''\subset D$ which is congruent to $\triangle P'Q'R'$,
$\triangle P''Q''R''$ is large.
\end{enumerate}
$\mu(\triangle PQR)$ is well defined from Theorem \ref{t7}.

\begin{rem}
 If $\triangle P_1P_2P_3\sim \triangle Q_1Q_2Q_3$, then it is assumed that $P_i$ corresponds to $Q_i$ for each $i=1,2,3$.   
\end{rem}

\begin{lem}\label{largelarge}
 Suppose that $\triangle PQR$ be included in $D$ and 
   for any $\triangle P_1Q_1R_1\subset D$ which is congruent to $\triangle PQR$,
$\triangle P_1Q_1R_1$ is large.
If $\triangle P_2Q_2R_2$ is included in $D$, $\triangle P_2Q_2R_2\sim \triangle PQR$, 
 and $\kappa(\triangle PQR)<\kappa(\triangle P_2Q_2R_2)$, then
 $\triangle P_2Q_2R_2$ is large.
\end{lem}
\begin{proof}
    We assume that $\triangle P_2Q_2R_2$ is included in $D$, $\triangle P_2Q_2R_2\sim \triangle PQR$, 
 and $\kappa(\triangle PQR)<\kappa(\triangle P_2Q_2R_2)$. We put $s=\kappa(\triangle PQR)/\kappa(\triangle P_2Q_2R_2)$.
With the center at $P_2Q_2$ serving as the focal point, we scale $\triangle P_2Q_2R_2$ by a factor of $s$, resulting in the transformation of the triangle into $\triangle P_3Q_3R_3$. Then, $\triangle P_3Q_3R_3$ is congruent to $\triangle PQR$ and it is included in $\triangle P_2Q_2R_2$.
We consider the ellipses $E_2$ and $E_3$ associated with points $P_2$ and $Q_2$, and $P_3$ and $Q_3$, respectively, as defined in Theorem \ref{t4}.
Since it holds that $d(P_2,Q_2)>d(P_3,Q_3)$, we have $\Delta(P_2,Q_2)<\Delta(P_3,Q_3)$.
Therefore, $E_2$ is included in $E_3$.
Since $\triangle P_3Q_3R_3$ is large, $R_3$ lies outside $E_3$.
We assume that $\triangle P_2Q_2R_2$ is not large.
Then, $R_2$ lies inside or on the boundary of the ellipse $E_2$.
Therefore, the interior of $\triangle P_2Q_2R_2$ is contained within the interior of $E_2$.
Hence, we see that $R_3$ lies inside $E_3$, which is a contradiction. 
\end{proof}

%$\mu(\triangle PQR)$ is well defined from Theorem \ref{t7}. 

We define
\begin{align}\label{triangletriangle}
\triangle :=\{\triangle PQR \subset D \mid \delta(P,Q,R)=\Delta(P,Q)\}.
\end{align}
From Theorem \ref{t1}, we observe that $\triangle PQR\in \triangle$
if and only if $\rho(\triangle PQR)=1/3$ and $\triangle PQR$ is not large.
\begin{rem}
 Precisely we define $\triangle$ as follows:
 \begin{align*}
\triangle :=\{(P,Q,R)\in D^3 \mid \delta(P,Q,R)=\Delta(P,Q)\}.
\end{align*}
\end{rem}

We have 
\begin{prop}\label{triangle PQR<1}
Let $\triangle PQR$ be a triangle in ${\mathbb R}^2$.
Then, $0<\mu(\triangle PQR)<1$.
\end{prop}
 \begin{proof}
 For any $\triangle P'Q'R'\subset D$, we see $\kappa(\triangle P'Q'R')<1$.
 Therefore, from Theorem \ref{t7}, 
 we have $\mu(\triangle PQR)<1$.
 Let $\triangle P'Q'R' \subset D$ be similar to  $\triangle PQR$.
% For $\lambda\in {\mathbb R}$, we define the similarity transformation $\mathcal{S}_{\lambda}$
% as for $X\in {\mathbb R}^2$
%\begin{align*}
%\mathcal{S}_{\lambda}(X):=\lambda X.
%\end{align*}
It is not difficult to see that
\begin{align*}
&\lim_{\lambda\to 0}\delta(\mathcal{S}_{\lambda}(P'),\mathcal{S}_{\lambda}(Q'),\mathcal{S}_{\lambda}(R'))=0,\\
&\lim_{\lambda\to 0}\Delta(\mathcal{S}_{\lambda}(P'),\mathcal{S}_{\lambda}(Q'))=\infty.
\end{align*}
Therefore, from Theorem \ref{t1}, there exists  $\lambda'>0$ such that
for every $\lambda$ with $0<\lambda<\lambda'$,
$\rho(\triangle \mathcal{S}_{\lambda}(P')\mathcal{S}_{\lambda}(Q')\mathcal{S}_{\lambda}(R'))>\frac{1}{3}$, which leads to $0<\mu(\triangle PQR)$.
 \end{proof}

We %have 
need some lemmas.

\begin{lem}\label{double ratio}
Let $x_i \in \mathbb{R}$ for $i=1,2,3,4$ such that $x_4 < x_1 < x_2 < x_3$.
Let $dr(x_1,x_2;x_3,x_4)$ be the double ratio
\begin{align*}
\dfrac{(x_3-x_1)(x_4-x_2)}{(x_3-x_2)(x_4-x_1)}.
\end{align*}
When keeping $x_2 - x_1$ constant, the value of $dr$ is minimized when $x_1$ and $x_2$ are symmetric with respect to the midpoint of $x_3$ and $x_4$.
\end{lem}

\begin{proof}
    We may assume that $x_3=1$ and $x_4=-1$ without loss of generality.
    We put  $L=(x_2-x_1)/2$. 
    Then, we have
\begin{align*}
&dr(x_1,x_1+2L;1,-1)-dr(-L,L;1,-1)=\\
&=\frac{4(L+x_1)^2L}{(1-2L-x_1)(L-1)^2(x_1+1)}\geq 0.
\end{align*}
The equality holds only when $x_1$ is equal to $-L$. 
Thus, we have the claim of the lemma.
\end{proof}

\begin{lem}\label{double ratio2}
Let $\{a_n^{(i)}\}_{n\geq 0}$ for $i=1,2,3,4$
be  sequences over $\mathbb{R}$ such that for all $n\geq 0$,
$a_n^{(4)}<a_n^{(1)}<a_n^{(2)}<a_n^{(3)}$ and
 $\lim_{n\to \infty}dr(a_n^{(1)},a_n^{(2)};a_n^{(3)},a_n^{(4)})=1$.
Then, 
\begin{align*}
\lim_{n\to \infty}\dfrac{a_n^{(2)}-a_n^{(1)}}{a_n^{(3)}-a_n^{(4)}}=0.
\end{align*}
\end{lem}
\begin{proof}
    We may assume that $a_n^{(3)}=1$ and $a_n^{(4)}=-1$ without loss of generality.
    We put  $L_n=(a_n^{(2)}-a_n^{(1)})/2$, $b_n^{(1)}=-L_n$ and $b_n^{(2)}=L_n$. 
    Then, we have $a_n^{(4)}<b_n^{(1)}<b_n^{(2)}<a_n^{(3)}$ and $b_n^{(2)}-b_n^{(1)}=a_n^{(2)}-a_n^{(1)}$.
    From Lemma \ref{double ratio}, for all $n\geq 0$, we have
   \begin{align*}
   dr(a_n^{(1)},a_n^{(2)};a_n^{(3)},a_n^{(4)})\geq  dr(b_n^{(1)},b_n^{(2)};a_n^{(3)},a_n^{(4)})\geq 1,
   \end{align*} 
   which implies $\lim_{n\to \infty}dr(b_n^{(1)},b_n^{(2)};a_n^{(3)},a_n^{(4)})=1$.
Since  it holds that for $n\geq 0$,
\begin{align*}
dr(b_n^{(1)},b_n^{(2)};a_n^{(3)},a_n^{(4)})=\dfrac{(1+L_n)^2}{(1-L_n)^2},
\end{align*}
we have $\lim_{n\to \infty}\frac{1+L_n}{1-L_n}=1$.
Putting for $n\geq 0$  $K_n=\frac{1+L_n}{1-L_n}$, we have $\lim_{n\to \infty}L_n=\frac{K_n-1}{K_n+1}=0$. Therefore, we have 
\begin{align*}
\lim_{n\to \infty}\dfrac{a_n^{(2)}-a_n^{(1)}}{a_n^{(3)}-a_n^{(4)}}=0.
\end{align*}
\end{proof}

\begin{lem}\label{alsogivenlemma}
Let $\triangle PQR$ be a triangle in $D$. The value of $\mu(\triangle PQR)$ is  given by any one of the equivalent quantities (1), (2), or (3).
\begin{align}\label{alsogiven}
&\text{(1)\ }\sup \{\kappa(\triangle P'Q'R')\mid \triangle P'Q'R'\sim \triangle PQR,\   \triangle P'Q'R' \in \triangle\},\\
&\text{(2)\ }\sup \{\kappa(\triangle P'Q'R')\mid \triangle P'Q'R'\sim \triangle PQR,\ \triangle P'Q'R'\subset D,  \triangle P'Q'R' \nonumber\\
&\text{ is not large.}\},\nonumber\\
&\text{(3)\ }\max \{\kappa(\triangle P'Q'R')\mid \triangle P'Q'R'\sim \triangle PQR,\   \triangle P'Q'R' \in \triangle\}.\nonumber
\end{align}
\end{lem}
\begin{proof}
The proof of (1). Let $M$ be the quantity referenced as (\ref{alsogiven}). 
We assume that $\triangle P'Q'R'\sim \triangle PQR$ and $\triangle P'Q'R' \in \triangle$.
Since $\triangle P'Q'R'$ is not large, from Lemma \ref{largelarge},
we have $\mu(\triangle PQR)\geq \kappa(\triangle P'Q'R')$.
%Let $\epsilon > 0$ be a sufficiently small number.
%Take the midpoint of $P'Q'$ and use that point as the center to scale  $\triangle P'Q'R'$ %by a factor of $1 - \epsilon > 0$ to obtain  $\triangle P''Q''R''$.
%We have $\kappa(\triangle P''Q''R'')=(1 - \epsilon)\kappa(\triangle P'Q'R')$.
%Since it holds $d(P',Q')>d(P'',Q'')$, we have  $\Delta(P',Q')<\Delta(P'',Q'')$.
%Using Theorem \ref{t4}, it is not difficult to observe that
%$\delta(P,Q,R)>\delta(P',Q',R')$.
%As $R''$ lies inside the ellipse $E$ related to $P', Q'$ in Theorem \ref{t4}, we can see %that $\delta(P'',Q'',R'')=\delta(P',Q',R'')<\Delta(P',Q')$.
%Therefore, we have $\delta(P'',Q'',R'')<\Delta(P'',Q'')$, which implies 
%$\rho(\triangle P''Q''R'')>\frac{1}{3}$.
%Since it holds $\lim_{\epsilon\to 0^+} \kappa(\triangle P''Q''R'')=\kappa(\triangle P'Q'R')$, we have $\mu(\triangle PQR)\geq \kappa(\triangle P'Q'R')$. 
Therefore, we have $\mu(\triangle PQR)\geq M$.
Let $\epsilon > 0$ be a sufficiently small number.
Then, there exists $\triangle P'Q'R' \subset D$ such that
$\triangle P'Q'R'\sim \triangle PQR$, 
$(1-\epsilon)\mu(\triangle PQR)=\kappa(\triangle P'Q'R')$,
and $\delta(P',Q',R')\leq\Delta(P',Q')$.
For $\nu \geq 1$, consider the triangle $\triangle P'_{\nu}Q'_{\nu}R'_{\nu}$ obtained by enlarging triangle $\triangle P'Q'R'$ around the midpoint of $P'Q'$ by a factor of $\nu$.
There exists $\nu_1>1$ such that 
$\triangle P'_{\nu_1}Q'_{\nu_1}R'_{\nu_1}$ is tangent to $S^1$.
First, we assume that $P'_{\nu_1}$ or $Q'_{\nu_1}$ lies on $S^1$.
Then, since $\lim_{\nu\to \nu_1^{-}}d(P'_{\nu},Q'_{\nu})=\infty$, we have $\lim_{\nu\to \nu_1^{-}}\Delta(P'_{\nu},Q'_{\nu})=0$.
It is not difficult to observe that for $\nu$ with $\nu_1>\nu\geq 1$, $\delta(P'_{\nu},Q'_{\nu},R'_{\nu})\geq\delta(P',Q',R')$ holds.
Therefore, there exists $\nu_2$ with $\nu_1>\nu_2\geq 1$
such that $\delta(P'_{\nu_2},Q'_{\nu_2},R'_{\nu_2})=\Delta(P'_{\nu_2},Q'_{\nu_2})$, which implies
that $\triangle P'_{\nu_2}Q'_{\nu_2}R'_{\nu_2}\in \triangle$.
Therefore, we have
$\kappa(\triangle P'_{\nu_2}Q'_{\nu_2}R'_{\nu_2})=\nu_2(1-\epsilon)\mu(\triangle PQR)$.
Next, we assume that $R'_{\nu_1}$ lies on $S^1$ and $P'_{\nu_1}$ and $Q'_{\nu_1}$ do not lie on $S^1$.
Then, since it holds that $0<\lim_{\nu\to \nu_1^{-}}
\Delta(P'_{\nu},Q'_{\nu})<\infty$
and $\lim_{\nu\to \nu_1^{-}}
\delta(P'_{\nu},Q'_{\nu},R'_{\nu})=\infty$, 
Therefore, there exists $\nu_2$ with $\nu_1>\nu_2\geq 1$
such that $\delta(P'_{\nu_2},Q'_{\nu_2},R'_{\nu_2})=\Delta(P'_{\nu_2},Q'_{\nu_2})$, which implies
that $\triangle P'_{\nu_2}Q'_{\nu_2}R'_{\nu_2}\in \triangle$.
Therefore, we have
$\kappa(\triangle P'_{\nu_2}Q'_{\nu_2}R'_{\nu_2})=\delta_2(1-\epsilon)\mu(\triangle PQR)$.
Hence, in either case, we obtain $M\geq \nu_2(1-\epsilon)\mu(\triangle PQR)$.
Considering $\epsilon\to 0$, we have
$M\geq \mu(\triangle PQR)$.

The proof of (2) is similar to that of (1). Thus, we omit the proof. 

The proof of (3).
Let $M$ be the quantity referenced as (\ref{alsogiven}).
From Proposition \ref{triangle PQR<1},  we have $0<M<1$.
There exists a sequence of triangles $\triangle P_nQ_nR_n$ such that
$\triangle P_nQ_nR_n\sim \triangle PQR$, 
$\triangle P_nQ_nR_n \in \triangle$, and 
$\lim_{n\to \infty}\kappa(\triangle P_nQ_nR_n)=M$.
Considering $(P_n,Q_n,R_n)\in (D\cup S^1)^3$ for $n\geq 0$, 
we observe that
there exists a subsequence $(P_{n_i},Q_{n_i},R_{n_i})$
which converges to a point $(P^*,Q^*,R^*)$ in $(D\cup S^1)^3$.
Since $\kappa$ is the continuous map, 
we observe that  $\triangle P^*Q^*R^*\sim \triangle PQR$ and $\kappa(\triangle P^*Q^*R^*)=M$.
If $\triangle P^*Q^*R^*$ is included in $D$, it is not difficult to see that
$\triangle P^*Q^*R^*\in \triangle$.
We assume that $\{P^*,Q^*,R^*\} \cap S^1\ne \emptyset$.
We assume that $P^* \in  S^1$
without loss of generality.
Then, we have $d(P^*,Q^*)=\infty$.
For each $n_i$, draw a hyperbolic perpendicular line from point $R_{n_i}$ to lines $P_{n_i}Q_{n_i}$, and designate the intersection point as $S_{i}$.
Let's denote $\nu>0$ as the distance in the Euclidean space between the lines $P^*Q^*$ and $R^*$.
Then, there exists $K>0$ such that
$|R_{n_i}S_{i}|>\nu/2$ for $i>K$.
Since $\lim_{i\to \infty}d(P_{n_i},Q_{n_i})=\infty$ and, $d(S_i,R_{n_i})=\Delta(P_{n_i},Q_{n_i})$ for $i\geq 0$, 
we have $\lim_{i\to \infty}d(S_i,R_{n_i})=0$.
From Lemma \ref{double ratio2}, 
we have $\lim_{i\to \infty}|R_{n_i}S_{i}|=0$, 
which contradicts that $|R_{n_i}S_{i}|>\nu/2$ for $i>K$.
Thus, we have the claim of (3).
\end{proof}

The following theorem becomes apparent from the definition of $\mu(\triangle PQR)$. However, when $\triangle PQR$ is geometrically large in the Euclidean sense, it becomes 'large' and
its geometric significance becomes more pronounced.

\begin{thm}
Let $\triangle PQR$ be an triangle in $D$. 
If $\kappa(\triangle PQR)>\mu(\triangle PQR)$, then $\triangle PQR$ is large; i.e.,
there exist two triangles inscribed in ${S^1}$ and 
circumscribing $\triangle PQR$.
\end{thm}

We give an example.

\begin{prop}\label{prop53}
\begin{figure}
    \centering
\begin{tikzpicture}
\draw(0,0) circle (2);
\draw (-1/2-0.4,1.732/2)--(-1/2-0.4,-1.732/2)--(1-0.4,0)--cycle;
 \draw (-1/2-0.4,1.732/2)node[left]{$P'$};
 \draw (-1/2-0.4,-1.732/2)node[left]{$Q'$};
  \draw (1-0.4,0)node[right]{$R'$};
 \draw (-1/2-0.4,0)node[left]{$S'$};
 
\draw[dashed] (-1/2-0.4,1.732/2)--(-1/2-0.4,0.893028554974588*2);
\draw[dashed] (-1/2-0.4,-1.732/2)--(-1/2-0.4,-0.893028554974588*2);
\draw[dashed] (2, 0)--(-2,0);
 \draw (-1/2-0.4,0.893028554974588*2)node[above]{$v_1$};
\draw (-1/2-0.4,-0.893028554974588*2)node[below]{$v_2$};
\draw (-2, 0)node[left]{$v_4$};
\draw (2,0)node[right]{$v_3$};
 \end{tikzpicture}
    \caption{Points in the proof of Proposition \ref{prop53}.}
    \label{fig:example3-2}
\end{figure}

Let $\triangle PQR$ be an equilateral triangle in ${\mathbb R}^2$.
Then, $\mu(\triangle PQR)=\frac{1}{2}$.
\end{prop}

\begin{proof}
We put  $P=(-\frac{1}{4}, \frac{\sqrt{3}}{4})$, $Q=(-\frac{1}{4}, -\frac{\sqrt{3}}{4})$,
 $R=(\frac{1}{2}, 0)$, and $S=(-\frac{1}{4}, 0)$.
Then, $\triangle PQR$ is an equilateral triangle and $\rho(\triangle PQR)=\frac{1}{3}$
 as seen in Example \ref{example1}.
 First, we consider $\mathcal{P}_{\tau_1,0}(\triangle PQR)$ for $\tau_1\in {\mathbb R}$.
 We note that $\mathcal{P}_{\tau_1,0}(\triangle PQR)\subset D$ implies
 $-\frac{\sqrt{13}-1}{4}<\tau_1<\frac{1}{2}$.
We set  $P'=\mathcal{P}_{\tau_1,0}(P)$, $Q'=\mathcal{Q}_{\tau_1,0}(Q)$,
$R'=\mathcal{P}_{\tau_1,0}(R)$, and $S'=\mathcal{P}_{\tau_1,0}(S)$ for $-\frac{\sqrt{13}-1}{4}<\tau_1<\frac{1}{2}$ (see Figure \ref{fig:example3-2}).
Let $u=-\frac{1}{4}+\tau_1$. 
Then, $P'=(u, \frac{\sqrt{3}}{4})$, $Q'=(u, -\frac{\sqrt{3}}{4})$,
$R'=(u+\frac{3}{4}, 0)$, and $S'=(u, 0)$.
We note $-\frac{\sqrt{13}}{4}<u<\frac{1}{4}$.
Furthermore, we put $v_1=(u,\sqrt{1-u^2})$, $v_2=(u,-\sqrt{1-u^2})$,
$v_3=(1,0)$, and $v_4=(-1,0)$. 
Then, we have
\begin{align*}
&d(P',Q')=\log \dfrac{\sqrt{1-u^2}+\sqrt{3}/4}{\sqrt{1-u^2}-\sqrt{3}/4},\\
&\Delta(P',Q')=\log\dfrac{e^{d(P,Q)}+1}{e^{d(P,Q)}-1}\\
&=\log\dfrac{\dfrac{\sqrt{1-u^2}+\sqrt{3}/4}{\sqrt{1-u^2}-\sqrt{3}/4}+1}{\dfrac{\sqrt{1-u^2}+\sqrt{3}/4}{\sqrt{1-u^2}-\sqrt{3}/4}-1}=\log\dfrac{4\sqrt{1-u^2}}{\sqrt{3}}.
\end{align*}
Since $\delta(P',Q',R')=d(R',S')$, we have
\begin{align*}
\delta(P',Q',R')=\dfrac{1}{2}\log \dfrac{(u+7/4)\cdot (1-u)}{(1/4-u)\cdot (u+1)}.
\end{align*}
We see
\begin{align*}
&\dfrac{(u+7/4)\cdot (1-u)}{(1/4-u)\cdot (u+1)}-\left(\dfrac{4\sqrt{1-u^2}}{\sqrt{3}}\right)^2\\
&=\dfrac{(1-u)(4u+1)^2(4u+5)}{3(u+1)(1-4u)}\geq 0,
\end{align*}
where the equality holds if and only if $u=-\frac{1}{4}$.
Therefore, we  see
\begin{align}\label{h(P',Q',R')}
\delta(P',Q',R')\geq \Delta(P',Q'),
\end{align}
where the equality holds if and only if  $\triangle P'Q'R'=\triangle PQR$.
Therefore, $\rho(\triangle P'Q'R')=\frac{1}{3}$.
Next, let us consider a translation parallel to the $y$-axis; i.e.,
we consider %$\mathcal{P'}_{0,\tau_2}(\triangle P'Q'R')$
$\mathcal{P}_{0,\tau_2}(\triangle P'Q'R')\left(=\mathcal{P}_{\tau_1,\tau_2}(\triangle PQR)\right)$ 
for $\tau_2\in {\mathbb R}$.
 We note that we only consider the case of $\tau_2>0$ by symmetry.
 Therefore, we assume $\tau_2>0$ and we have
 $0<\tau_2<\sqrt{1-u^2}-\frac{\sqrt{3}}{4}$.
 We set  $P''=\mathcal{P}_{0,\tau_2}(P')$, $Q''=\mathcal{P}_{0,\tau_2}(Q')$,
$R''=\mathcal{P}_{0,\tau_2}(R')$, and $S''=\mathcal{P}_{0,\tau_2}(S')$.
 
 Then, we have
 \begin{align}\label{d(P'',Q'')}
 d(P'',Q'')=\dfrac{1}{2}\log \dfrac{(\sqrt{1-u^2}+\sqrt{3}/4+\tau_2)(\sqrt{1-u^2}+\sqrt{3}/4-\tau_2)}{(\sqrt{1-u^2}-\sqrt{3}/4+\tau_2)(\sqrt{1-u^2}-\sqrt{3}/4-\tau_2)}.
 \end{align}
 In general, if $A>B>x>0$ for $A,B,x\in \mathbb{R}$, we have
 $\frac{(A+x)(A-x)}{(B+x)(B-x)}> \frac{A^2}{B^2}$.
 Therefore, from (\ref{d(P'',Q'')}), we have
 $d(P'',Q'')>d(P',Q')$, which implies
 \begin{align}\label{Delta'(P'',Q'')<}
 \Delta(P'',Q'')<\Delta(P',Q').
 \end{align}
Let $E$ (resp., $E'$) be the ellipse related to $P',Q'$(resp., $P'',Q''$) in Theorem \ref{t4}.
Let $\Phi$ (resp., $\Phi'$) denote the inner region bounded by $E$ (resp., $E'$).
We can see from (\ref{h(P',Q',R')}) that $R'$ is not included in $\Phi$.
Because $E$ is an ellipse with its center in the $x$-axis and  one of its axes  parallel to the $y$-axis, 
we obtain  
$\Phi\subset \{(x,y)\in {\mathbb R}^2\mid x<u+\frac{3}{4} \}$. 
From the fact that line $P'Q'$ is equal to line $P''Q''$
and (\ref{Delta'(P'',Q'')<}),
we have  $\Phi' \subset \Phi$. Therefore, $R''=(u+\frac{3}{4}, \tau_2)$
is not included in $\Phi'$.
Therefore, $\rho(\triangle P''Q''R'')=\frac{1}{3}$.
Thus, we show that the rotation number %ewlated
related 
to $\triangle PQR$ is invariant under translation, which means that it is invariant under congruent transformations by symmetry.
It is observable that if we reduce $\triangle PQR$  even slightly closer to the origin, the related  rotation number is  $>\frac{1}{3}$.
As a result, the proof is complete.
\end{proof}

\subsection{Evaluation of $\mu$}
The aim of this Section is to demonstrate that $\mu(\triangle PQR)\geq \frac{1}{2}$.
From Lemma \ref{alsogivenlemma}, we have the following lemma.
\begin{lem}\label{l5-1}
Let $\triangle PQR\in \triangle$.
Then, $\mu(\triangle PQR)\geq \kappa(\triangle PQR)$.
\end{lem}

Using Lemma \ref{l5-1}, %let's 
let us explore the lower bound of $\mu(\triangle PQR)$.

For $(u,v)\in D$ with $u>0$,
we define
\begin{align*}
%&\triangle(u,v):=\{\triangle PQR \mid \triangle PQR\in \triangle, %\text{first coordinate of $R\geq v$}\},\\
&\triangle(u,v):=\{\triangle PQR \in \triangle \mid P=(-u,v), Q=(u,v), R=(x,y) \text{ with } y\ge v\},\\
&\triangle_v:=\bigcup_{0<u<\sqrt{1-v^2}}\triangle(u,v)%,
.
\end{align*}
%where $P=(-u,v)$ and $Q=(u,v)$.

We will leave the proof of the following lemma to the reader.
\begin{lem}\label{l5-2}
Let $|v|<1$.
Let $\triangle ABC\subset D$.
Then, there exists $\triangle PQR\in \triangle_v$
such that 
$\triangle PQR\sim \triangle ABC$.
\end{lem}

Now, we will show that if $\triangle PQR$ is an  isosceles triangle, then $\mu(\triangle PQR)\geq \frac{1}{2}$.

Let $0<u<\frac{\sqrt{15}}{4}$.
$P=(-u,-\frac{1}{4}), Q=(u,-\frac{1}{4})$.

Let $\triangle PQR \in \triangle(u,-\frac{1}{4})$.
From Theorem \ref{t4}, we observe that $(x,y)$ satisfies
\begin{align}\label{16u^2+15}
\dfrac{x^2}{B^2}+\dfrac{(y-C)^2}{A^2}=1,
\end{align}
where
\begin{align}\label{ABC}
&A=\dfrac{(16u^2+15)(15-16u^2)}{(16u^2+25)(16u^2+9)},\nonumber\\
&B=\dfrac{(16u^2+15)}{\sqrt{(16u^2+25)(16u^2+9)}},\\
&C=-\dfrac{256u^2}{(16u^2+25)(16u^2+9)}.\nonumber
\end{align}
Then, $R=\left(0,\frac{9-16u^2}{9+16u^2}\right)$ is on the ellipse (\ref{16u^2+15}).

\begin{figure}
    \centering
    \begin{tikzpicture}
\draw(0,0) circle (2);

\draw (0.6*2,-0.25*2)node[above]{$Q$};
\draw (-0.6*2,-0.25*2)node[above]{$P$};
\draw(0,-0.202987725*2) circle (0.9742961914*2 and 0.42249992*2);

\draw (0.968245*2,-0.5)--(-0.968245*2,-0.5);

\coordinate (P) at (-0.6*2,-0.25*2);
\fill (P) circle [radius=1.5pt];
\coordinate (Q) at (0.6*2,-0.25*2);
\fill (Q) circle [radius=1.5pt];
\draw[dashed] (0,2)--(0,-2);
\draw[dashed] (2,0)--(-2,0);

 \end{tikzpicture}

    \caption{Ellipse (\ref{16u^2+15}).}
    \label{fig:EllipseR}
\end{figure}

\begin{lem}\label{l5-8}
Let $0<u<\frac{\sqrt{15}}{4}$.
$P=(-u,-\frac{1}{4}), Q=(u,-\frac{1}{4})$ and 
$R=\left(0,\frac{9-16u^2}{9+16u^2}\right)$.
When varying the value of $u$, $\kappa(\triangle PQR)$ reaches its minimum 
$\frac{1}{2}$
at $u = \frac{\sqrt{3}}{4}$. 
\end{lem}
\begin{proof}
   Let $S=(0,-\frac{1}{4})$.
Solving 
\begin{align*}
\dfrac{|RS|}{|PS|}=\dfrac{3(15-16u^2)}{4u(9+16u^2)}%\geq 
>1,
\end{align*}
we observe that  $\angle PRQ < %\dfrac{\pi}{2}
\frac{\pi}{2}$ if and only if
$u < \omega_2(=0.55021\ldots)$, where
$\omega_2\in [0,1]$ is the real root of $64x^3+48x^2+36x-45=0$. 
Therefore,  we have
\begin{align}
\kappa(\triangle PQR)=\begin{cases}
    \dfrac{4096u^6 + 6912u^4 - 3024u^2 + 2025}{24(15-16u^2)(9+16u^2)}& \text{if $0<u<\omega_2$},\\
    %$u$
    u& \text{if $\omega_2 \leq u< \frac{\sqrt{15}}{4}$}.
\end{cases}
\end{align}
We assume that $0<u<\omega_2$.
Then, we have 
\begin{align}\label{(45+16u^2)(-3+16u^2)}
\kappa(\triangle PQR)-\frac{1}{2}=\dfrac{(45+16u^2)(-3+16u^2)^2}{24(15-16u^2)(9+16u^2)}\geq 0.
\end{align}
Equality holds for the inequality (\ref{(45+16u^2)(-3+16u^2)}) when $u = \frac{\sqrt{3}}{4}$.
Next, we assume that $\omega_2 \leq u< \frac{\sqrt{15}}{4}$.
Then, we have
\begin{align*}
\kappa(\triangle PQR)=u\geq \omega_2>\dfrac{1}{2}.
\end{align*}
Thus, we have the claim of the lemma.
\end{proof}

\begin{lem}\label{isosceles}
 Let $\triangle PQR$ %is
 be 
 an  isosceles triangle in $D$.
 Then $\mu(\triangle PQR)\geq \frac{1}{2}$ and  equality holds
 when  $\triangle PQR$ is an equilateral triangle.
\end{lem}

\begin{proof}
Let $0<u<\frac{\sqrt{15}}{4}$.
$P=(-u,-\frac{1}{4}), Q=(u,-\frac{1}{4})$ and 
$R=\left(0,\frac{9-16u^2}{9+16u^2}\right)$.
   Then, we have $\dfrac{|RS|}{|PS|}=\dfrac{3(15-16u^2)}{4u(9+16u^2)}$.
Therefore, we have
\begin{align*}
\lim_{u\to 0^+}\dfrac{|RS|}{|PS|}=\infty, \text{ and } \lim_{u\to \frac{\sqrt{15}}{4}^{-}}\dfrac{|RS|}{|PS|}=0.
\end{align*}
Therefore, for any isosceles triangle, we can select some $0 < u < \frac{\sqrt{15}}{4}$ such %taht
that 
it  becomes similar to $\triangle PQR$.
Therefore, 
from Proposition \ref{prop53}, Lemma \ref{l5-1}, and \ref{l5-8},  we have the claim of the theorem.
\end{proof}

Let $0<u<\frac{\sqrt{15}}{4}$, %and
$P=(-u,-\frac{1}{4}), Q=(u,%\frac{1}{4}
-\frac{1}{4}
)$, and
%Let $0\leq x\leq B$, where 
let $B$ be %is 
defined as in (\ref{ABC}).
%Let
If we take
\begin{align}\label{RRR}
R=\left(x,-\frac{(16u^2 - 15)\sqrt{-(16u^2 + 25)(16u^2 + 9)x^2+(16u^2 + 15)^2}+ 256u^2}{(16u^2 + 25)(16u^2 + 9)}\right).
\end{align}
for $0\leq x \leq B$,
%Then,
then $\triangle PQR \in \triangle(u,-\frac{1}{4})$
and $R$ lies 
on  the upper half of the ellipse (\ref{16u^2+15}).
%Let $\frac{\sqrt{15}}{4}\leq x\leq B$.
On the other hand, if we take
%Let
\begin{align}\label{RRR2}
R=\left(x,-\frac{(15-16u^2)\sqrt{-(16u^2 + 25)(16u^2 + 9)x^2+(16u^2 + 15)^2}+ 256u^2}{(16u^2 + 25)(16u^2 + 9)}\right).
\end{align}
for $\frac{\sqrt{15}}{4}\leq x\leq B$,
%Then,
then $R$ is in $\triangle(u,-\frac{1}{4})$
and  
lies on  the lower half of the ellipse (\ref{16u^2+15}).

\begin{lem}\label{l5-9}
Let $0<u<\frac{\sqrt{15}}{4}$.
Let $0\leq x\leq B$, where $B$ is defined as in (\ref{ABC}).
Let $P=(-u,-\frac{1}{4})$ and $Q=(u,-\frac{1}{4})$, with $R$ being defined as in formula (\ref{RRR}).
If $x\geq u$, then $\kappa(\triangle PQR)>%\dfrac{1}{2}
\frac{1}{2}$.
\end{lem}
\begin{proof}
We assume $x\geq u$.
From the fact that $\angle PQR\geq %\dfrac{\pi}{2}
\frac{\pi}{2}$, we have
\begin{align}\label{T_1+T_2}
4\kappa(\triangle PQR)^2-1=|PR|^2-1=\dfrac{T_1+T_2}{16(16u^2 + 25)^2(16u^2 + 9)^2},
\end{align}
where
\begin{align*}
&T_1=2097152xu^9 + 1048576u^{10} + 4194304x^2u^6 +8912896xu^7 + 4521984u^8+\\
&+8912896x^2u^4+ 13156352xu^5 + 1875968u^6 + 3686400x^2u^2+7833600xu^3-\\
&-4158976u^4 + 1620000xu - 3322800u^2 + 50625,\\
&T_2=8(15-16u^2)^3\sqrt{-(16u^2 + 25)(16u^2 + 9)x^2+(16u^2 + 15)^2}.
\end{align*}
We observe that $T_1$ increases as $x$ increases and $T_2>0$. We substitute $x$ with $u$ in $T_1$ and denote the resulting formula by $T_3$.
Then, we have $T_1 \geq T_3$.
We have
\begin{align*}
&T_3=3145728u^{10}+17629184u^8+23945216u^6+7361024u^4-1702800u^2+\\
&+ 50625>7361024u^4-1702800u^2+ 50625
\end{align*}
We observe that if $u\leq 0.18$ or $u\geq 0.45$, then
$7361024u^4-1702800u^2+ 50625>0$.
Therefore, if $u\leq 0.18$ or $u\geq 0.45$, then
from the formula (\ref{T_1+T_2}), we have $\kappa(\triangle PQR)>\frac{1}{2}$.

%We assume that $0.18<u<0.45$.
In the following, we consider the case $0.18<u<0.45$.
We first assume that $x\geq 0.4$.
We substitute $x$ with $\frac{2}{5}$ in $T_1$ and denote the resulting formula by $T_4$.
Then, 
\begin{align*}
&T_4=1048576u^{10} + \frac{4194304}{5}u^9 + 4521984u^8 + \frac{17825792}{5}u^7 + \frac{63676416}{25}u^6 +\\ &+\frac{26312704}{5}u^5 - \frac{68322816}{25}u^4 + 3133440u^3 - 2732976u^2 + 648000u + 50625
\end{align*}
Then, we have
\begin{align*}
&T_1\geq T_4>\\
&\frac{26312704}{5}u^5 - \frac{68322816}{25}u^4 + 3133440u^3 - 2732976u^2 + 648000u + 50625=\\
&=\frac{64}{25}(2055680u^2 - 1067544u + 140625)u^3+
2773440u^3 - 2732976u^2 + \\
&+648000u + 50625,
\end{align*}
where it is not difficult to see that $2055680u^2 - 1067544u + 140625>0$ and
$2773440u^3 - 2732976u^2 +648000u + 50625>0$.
Therefore, we have $T_1>0$.
Next, we assume that $x<0.4$.
Note that we only have to consider the case $0.18 < u < 0.4$, since $x \geq u$ holds.
%Since $x\geq u$ holds, if $0.4\leq u<0.45$, then we observe that $T_1>0$.
%Next, we assume that $0.18<u<0.4$ and $u\leq x<0.4$. 
Then, we have
\begin{align*}
&T_2>8(15-16u^2)^3\sqrt{-(16u^2 + 25)(16u^2 + 9)(2/5)^2+(16u^2 + 15)^2}\\
&=\dfrac{8}{5}(15-16u^2)^3\sqrt{5376u^4 + 9824u^2 + 4725}>8\cdot 13 (15-16u^2)^3.
\end{align*}
Therefore, we have
\begin{align*}
&T_1+T_2>T_3+8\cdot 13 (15-16u^2)^3=\\
&=3145728u^{10} + 17629184u^8 + 23519232u^6 + 8559104u^4 - 2826000u^2 +\\ &+401625>8559104u^4 - 2826000u^2 + 401625,
\end{align*}
where it is not difficult to see that $8559104u^4 - 2826000u^2 + 401625>0$.
Thus, we have the claim of the lemma.
\end{proof}

For some $R$ defined by formula (\ref{RRR2}), there exists $\triangle PQR$  such that $\kappa(\triangle PQR) < \frac{1}{2}$. Therefore, we aim to extract the characteristics of triangle $\triangle PQR$ related  to formula (\ref{RRR2}) and demonstrate $\kappa > \frac{1}{2}$ in a similar triangle related to the ellipse described in formula (\ref{daen-2}).

\begin{lem}\label{l5-10}
Let $0<u<\frac{\sqrt{15}}{4}$.
Let $\frac{\sqrt{15}}{4}\leq x\leq B$, where $B$ is in (\ref{ABC}).
Let $S=(0,-\frac{1}{4})$  with $R$ being defined as in formula (\ref{RRR2}).
Then, the slope of $SR$ is %less than or equal to $\dfrac{1}{4}
less than or equal to $\frac{1}{4}$.
\end{lem}
\begin{proof}
The slope of $SR$ is clearly maximized when $R = (B, C)$, and in that case, the slope is given by
\begin{align*}
&\frac{C + \frac{1}{4}}{B}=\dfrac{(16u^2-15)^2}{4(16u^2+15)\sqrt{(16u^2+25)(16u^2+9)}}=\\
&=\dfrac{1}{4}\cdot\dfrac{15-16u^2}{15+16u^2}\cdot\dfrac{\sqrt{256u^4-480u^2+225}}{\sqrt{256u^4+544u^2+225}}
\leq \dfrac{1}{4}.
\end{align*}    
\end{proof}

Let $0<u<1$, $P=(-u,0)$ and $Q=(u,0)$.
Let $\triangle PQR \in \triangle(u,0)$.
From Theorem \ref{t4}, we observe that $R=(x,y)$ satisfies 
\begin{align}\label{daen-2}
x^2+\left(\dfrac{1+u^2}{1-u^2}y\right)^2=1.
\end{align}
Therefore, we have
\begin{align}\label{RRda-2}
R=\left(x,\dfrac{(1-u^2)\sqrt{1-x^2}}{(1+u^2)}\right).
\end{align}

\begin{lem}\label{l5-11}
Let $0<u\leq x<1$, $P=(-u,0)$, $Q=(u,0)$ and
$R=\left(x,\frac{(1-u^2)\sqrt{1-x^2}}{(1+u^2)}\right)$.
If the slope of $OR$ is less than or equal to $\frac{1}{4}$,
then it holds that $\kappa(\triangle PQR)>\frac{1}{2}$.
\end{lem}
\begin{proof}
    If $u>\frac{1}{2}$ holds, then we have $|PR|>|PQ|>1$, which
    implies that $\kappa(\triangle PQR)>\frac{1}{2}$.
    Therefore, we assume that $0<u\leq \frac{1}{2}$.
    Now, we assume that the slope of $OR$ is equal to $\frac{1}{4}$.
    Then, from the equation $\frac{x}{4}=\frac{(1-u^2)\sqrt{1-x^2}}{(1+u^2)}$,
    we have
    \begin{align*}
    &x=\dfrac{4(1-u^2)}{\sqrt{17u^4-30u^2+17}}   =\dfrac{4}{5}\cdot \dfrac{5(1-u^2)}{\sqrt{17u^4-30u^2+17}}=\\
    &=\dfrac{4}{5}\cdot \dfrac{\sqrt{25u^4-30u^2+17+(8-20u^2)}}{\sqrt{17u^4-30u^2+17}}
    >\dfrac{4}{5}.
    \end{align*}
    Therefore,  we observe
    if the slope of $OR$ is less than or equal to $\frac{1}{4}$, 
    it holds that $x>\frac{4}{5}$.
    We have
\begin{align}\label{u^6+2u^5x+3u^4-2}
|PR|=\frac{\sqrt{u^6+2u^5x+3u^4+4u^3x+4u^2x^2-u^2+2ux+1}}{(u^2+1)}.
\end{align}
    Therefore, if $x=\frac{4}{5}$ holds, we have
    \begin{align*}
    |PR|^2-1=%\dfrac{25u^5+40u^4+50u^3+80u^2+11(1-u)+29}
    \frac{u(25u^5+40u^4+50u^3+80u^2+11(1-u)+29)}{25(u^2+1)^2}>0.
    \end{align*}
    Therefore, from the equation (\ref{u^6+2u^5x+3u^4-2}), if $x\geq \frac{4}{5}$, we have
    $|PR|>1$, which implies that $\kappa(\triangle PQR)>\frac{1}{2}$.
\end{proof}

\begin{thm}\label{obtuse}
Let $\triangle ABC$ be an obtuse triangle or a right triangle in $D$.
 Then,   $\mu(\triangle ABC)>\frac{1}{2}$.
\end{thm}
\begin{proof}
From Lemma \ref{l5-2}, there exists $\triangle PQR\in \triangle_{-\frac{1}{4}}$
such that $\triangle ABC \sim \triangle PQR$.
If $R$ is defined by formula (\ref{RRR}), then, by Lemma \ref{l5-9}, we have $\kappa(\triangle PQR) > \frac{1}{2}$.
Therefore, we have $\mu(\triangle ABC)>\frac{1}{2}$.

We assume that  $R$ is defined by formula (\ref{RRR2}).
From Lemma \ref{l5-2}, there exists $\triangle P'Q'R'\in \triangle_{0}$
such that $\triangle ABC \sim \triangle P'Q'R'$.
From Lemma \ref{l5-10} and \ref{l5-11}, we obtain that
$\kappa(\triangle P'Q'R') > \frac{1}{2}$.
Hence, we have $\mu(\triangle ABC)>\frac{1}{2}$.
\end{proof}

\begin{lem}\label{symmetrically}
Let $E$ be an ellipse.
Suppose two points $P, Q$ inside $E$ lie  symmetrically with respect to the minor axis of $E$.
Then, among the points $R$ on $E$, the circumradius of $\triangle PQR$ takes its minima when $RP=RQ$.
\end{lem}

\begin{proof}
Suppose $E$ is defined by the equation $\frac{x^2}{a^2}+\frac{y^2}{b^2}=1$, with $a>b>0$, and $P=(-u,h), Q=(u,h)$.
Since $P$ and $Q$ are inside $E$, $u$ and $h$ satisfy the inequality
\begin{equation}\label{uh-ineq}
\frac{u^2}{a^2}+\frac{h^2}{b^2}-1<0.
\end{equation}
A circle passing through $P, Q$ has the center on the $y$-axis.
By putting $(0,t)$ as center, we can express such a circle by 
\[ \Gamma_t \colon x^2+(y-t)^2=(t-h)^2+u^2. \]
Note that the radius of $\Gamma_t$ is equal to $r=\sqrt{(t-h)^2+u^2}$, so it takes minima when $|t-h|$ does.

Let $R=(a\cos\theta, b\sin\theta)$ be a point on $E$.
Since 
\[ t=\frac{h^2+u^2-x^2-y^2}{2
   (h-y)},\]
$\Gamma_t$ is the circumcircle of $\triangle PQR$ when
\[ t=\frac{h^2+u^2-a^2+(a^2-b^2) \sin ^2\theta}{2 (h-b \sin \theta)},\]
where $t$ goes to $\pm\infty$ when $\sin\theta=\frac{h}{b}$.
Then
\[\frac{dt}{d\theta}=\frac{\cos \theta \left[ -b\left(a^2-b^2\right) \sin ^2\theta+2 h
   \left(a^2-b^2\right) \sin \theta+b
   \left(-a^2+h^2+u^2\right)\right]}{2 (h-b \sin \theta)^2}\]
vanishes only when $\theta=\pm\frac{\pi}{2}+2n\pi$, because the part
\[ -b\left(a^2-b^2\right) \sin ^2\theta+2 h
   \left(a^2-b^2\right) \sin \theta+b
   \left(-a^2+h^2+u^2\right)\]
as a quadratic of $X=\sin\theta$ has the discriminant
\begin{align*}
Di/4 &=h^2\left(a^2-b^2\right)^2+b^2\left(a^2-b^2\right)\left(-a^2+h^2+u^2\right)\\
%&=\left(a^2-b^2\right) \left(a^2h^2+b^2u^2-a^2 b^2\right)\\
&=a^2b^2\left(a^2-b^2\right)\left(\frac{u^2}{a^2}+\frac{h^2}{b^2}-1\right) <0.
\end{align*}
Thus we see that $t$ takes extreme values when $\theta=\pm\frac{\pi}{2}$.
It is easy to see that the circumradius of $\triangle PQR$ takes the minima when $R=(0, \pm b)$.
\end{proof}

\begin{thm}\label{maintheorem}
 Let $\triangle PQR$ %is
 be 
 an triangle in $D$.
 Then $\mu(\triangle PQR)\geq \frac{1}{2}$ and  equality holds
 when  $\triangle PQR$ is an equilateral triangle.
\end{thm}

\begin{proof}
Let $0<u<\frac{\sqrt{15}}{4}$.
$P=(-u,-\frac{1}{4}), Q=(u,-\frac{1}{4})$.
Let $\triangle PQR \in \triangle(u,-\frac{1}{4})$ and be an acute triangle and not an  isosceles triangle.
We put $R'=\left(0,\frac{9-16u^2}{9+16u^2}\right)$.
Then, $\triangle PQR' \in \triangle(u,-\frac{1}{4})$ and 
$\triangle PQR'$ is an  isosceles triangle.
First, we assume that $\triangle PQR'$ is the acute triangle.
Then, from Lemma \ref{symmetrically}, we have
$\kappa(\triangle PQR)>\kappa(\triangle PQR')$.
Since from Lemma \ref{isosceles},  $\kappa(\triangle PQR')\geq \frac{1}{2}$ holds, 
we have $\kappa(\triangle PQR)>\frac{1}{2}$.
Next, we assume that $\triangle PQR'$ is the obtuse triangle or the right triangle.
Then, from the proof of Lemma \ref{l5-8}, we observe that 
$u\geq \omega_2(=0.55021\ldots)$. 
Since the circumradius of $\triangle PQR$ is equal to $u/\sin \angle PRQ$, 
it is greater than $\frac{1}{2}$.
Therefore,  we have
$\kappa(\triangle PQR)>\frac{1}{2}$.
Thus, from Lemma \ref{isosceles} and Theorem \ref{obtuse}, we have the claim of the theorem.
\end{proof}

\subsection{Explicit formula}
In this section, we provide an explicit formula for $\mu(\triangle PQR)$ when
%$\mu(\triangle PQR)$ 
$\triangle PQR$ is an  isosceles triangle.

\begin{lem}\label{maxisosceles}
Let $\triangle PQR$ be an isosceles triangle with $R$ as the vertex, and
$\triangle PQR\in\triangle$.
Then, there exists an isosceles triangle $\triangle P'Q'R'\in\triangle$
with $R'$ as the vertex such that $\triangle PQR\sim \triangle P'Q'R'$, 
$|OP'|=|OQ'|$, and  $\kappa(\triangle PQR)\leq \kappa(\triangle P'Q'R')$.
\end{lem}
\begin{proof}
We assume that $|OP'|\ne|OQ'|$. 
 Let $v_1, v_2\in {S^1}$ be two points at infinity, which intersect the hyperbolic line $PQ$ with $|Pv_1|<|Qv_1|$.
Let $S$ be the the midpoint of $P$ and $Q$.
Let $P',\ Q'$ be the points on the line $PQ$ such that
$|P'Q'|=|PQ|$ and $P'$ and $Q'$ are symmetric with respect to $S$.
Let $\triangle P'Q'R'$ be congruent to $\triangle PQR$, 
with $R$ and $R'$ positioned on the same side with respect to the line $PQ$.
Since $\triangle P'Q'R'$ is the isosceles
triangle, $\triangle P'Q'R'$ is included in $D$.
Since $d(P,Q)=\frac{1}{2}\log (\frac{|Pv2||Qv_1|}{|Pv1||Qv_2|})$, 
from Lemma \ref{double ratio}, we have 
$d(P,Q)> d(P',Q')$.
Therefore, we have $\Delta(P,Q)<\Delta(P',Q')$, which implies that
$\delta(P',Q',R')<\Delta(P',Q')$.
Hence, we observe that 
there exists a certain $\eta > 1$  such that, upon enlarging the triangle $\triangle P'Q'R'$ around the center $S$ by a factor of $\eta$, the resulting triangle $\triangle P''Q''R''$ belongs to  $\triangle$.
Thus, we have $\kappa(\triangle PQR)< \kappa(\triangle P''Q''R'')$.
 
\end{proof}
For $s>0$, we define 
\begin{align*}
q_s(X)=(s^4+ 2s^2+1)X^3+(- 6s^3+10s)X^2+(12s^2-1)X-8s.
\end{align*}
The polynomial $q_s(X)$ plays an important role in stating a theorem.

\begin{lem}
The polynomial $q_s(X)$ has a sole root in $(0,1)$.
\end{lem}
\begin{proof}
It holds that
    $q_s(0)=-8s<0$ and
    $q_s(1)=s^2(s-3)^2+5s^2+2s>0$.
Therefore, there exist either one or three roots in the interval $(0,1)$.    The discriminant of $q_s(X)$ is $-4(27s^2 - 1)^3$.
    Therefore, if $|s|>\frac{1}{3\sqrt{3}}$, then $q_s(X)$ has a unique root in $(0,1)$.
    We assume that $|s|\leq \frac{1}{3\sqrt{3}}$ and $q_s(X)$ has three roots in $(0,1)$. We put $\alpha_1, \alpha_2, \alpha_3$ as the roots of $q_s(X)$.
    Then, we have
    \begin{align*}
    \alpha_1+\alpha_2+\alpha_3=\dfrac{(6s^2-10)s}{s^4+ 2s^2+1}
    \leq \dfrac{(6/27-10)s}{s^4+ 2s^2+1}<0,
    \end{align*}
   which contradicts that $\alpha_i\in (0,1)$ for $i=1,2,3$.
   Therefore, we have the claim of the lemma.
\end{proof}

For $s>0$, we denote the root of $q_s(X)$ included in $(0,1)$
by $\xi(s)$.

\begin{thm}
Let $\triangle PQR$ be an isosceles triangle with $R$ as the vertex.
We set $s=\frac{|RS|}{|QS|}$, where $S$ is the midpoint of $PQ$.
Then, 
\begin{align*}
\mu(\triangle PQR)=\begin{cases}\xi(s)&s\leq 1,\\
\dfrac{(s^2+1)\xi(s)}{2s}& 1<s.\\
\end{cases}
\end{align*}
\begin{proof}
From Lemma \ref{alsogivenlemma} and Lemma \ref{maxisosceles},
we observe that 
\begin{align*}
\mu(\triangle PQR)=\max\{\kappa(\triangle P'Q'R')|\triangle P'Q'R'\in \bigcup_{v\in(-1,1)}\triangle_v,  \triangle P'Q'R'\sim\triangle PQR\}.
\end{align*}
For $(u,v)$ with $u^2+v^2<1$ and $u>0$,
let $\triangle P'Q'R' \in \triangle(u,v)$
and $\triangle P'Q'R'\sim\triangle PQR$.
From Theorem \ref{t4}, we observe that $(x,y)$ satisfies
\begin{align*}
\dfrac{x^2}{B^2}+\dfrac{(y-C)^2}{A^2}=1,
\end{align*}
where
\begin{align*}
&A=\dfrac{(v^2 - u^2 - 1)(u^2+v^2-1)}{(v^2 + u^2 + 2v + 1)(v^2 + u^2 - 2v + 1)},\\
&B=\dfrac{(1+u^2-v^2)\sqrt{1-v^2}}{\sqrt{-((v + 1)u^2 + (v^2 - 1)v - v^2 + 1)((v - 1)u^2 + (v^2 - 1)v + v^2 - 1)}},\\
&C=\dfrac{4vu^2}{(v^2 + u^2 + 2v + 1)(v^2 + u^2 - 2v + 1)}.
\end{align*}
First, we assume that $s\leq 1$.
Then, we have $\kappa(\triangle P'Q'R')=u$.
Therefore,  we have
\begin{align}\label{Lagrange}
\mu(\triangle PQR)=\max \left\{u\; \middle|\; \dfrac{A+C-t}{u}=s, u^2+v^2<1,u>0 \right\}. 
\end{align}

Then, by using the method of Lagrange multipliers, 
there exists $\lambda$ such that
the solution of (\ref{Lagrange}) satisfies the following equations:
\begin{align*}
&1=\lambda\left(\dfrac{A+C-t}{u}\right)_u,\\
&0=\lambda\left(\dfrac{A+C-t}{u}\right)_v,\\
&s=\dfrac{A+C-t}{u}.
\end{align*}
By eliminating $\lambda$, we obtain
\begin{align}
&v^4 + 2v^2u^2 + u^4 + 4v^3 + 6v^2 - 2u^2 + 4v + 1=0,\label{v^4 + 2v^2u^2}\\
&sv^2u + su^3 + v^3 + 2svu + vu^2 + v^2 + su + u^2 - v - 1=0\label{sv^2u + su^3}.
\end{align}
The left-hand side of equation (\ref{v^4 + 2v^2u^2}) is denoted as $f$, and the left-hand side of equation (\ref{sv^2u + su^3}) is denoted as $g$.
We consider $u$, $v$, and $s$ as variables, and by examining the Gröbner basis of the ideal generated by $f$ and $g$ in $\mathbb{Q}[u, v, s]$, we obtain the following:
\begin{align*}
(s^4+ 2s^2+1)u^3+(- 6s^3+10s)u^2+(12s^2-1)u-8s=0.
\end{align*}
Since it holds $0<u<1$, we have $u=\xi(s)$.
Therefore, we have $\mu(\triangle PQR)=\xi(s)$.
The proof can be similarly extended to the case of $s > 1$, and thus, the details of the proof are omitted.
\end{proof}
   
\end{thm}
%The following figure
Figure \ref{fig:isocelesmu} depicts the graph of $\mu(\triangle PQR)$ for an isosceles triangle $\triangle PQR$ with $R$ as the vertex, where $s = \frac{|RS|}{|QS|}$ and $S$ is the midpoint of $PQ$.

\begin{figure}[h]
\centering
\includegraphics[width=12truecm]{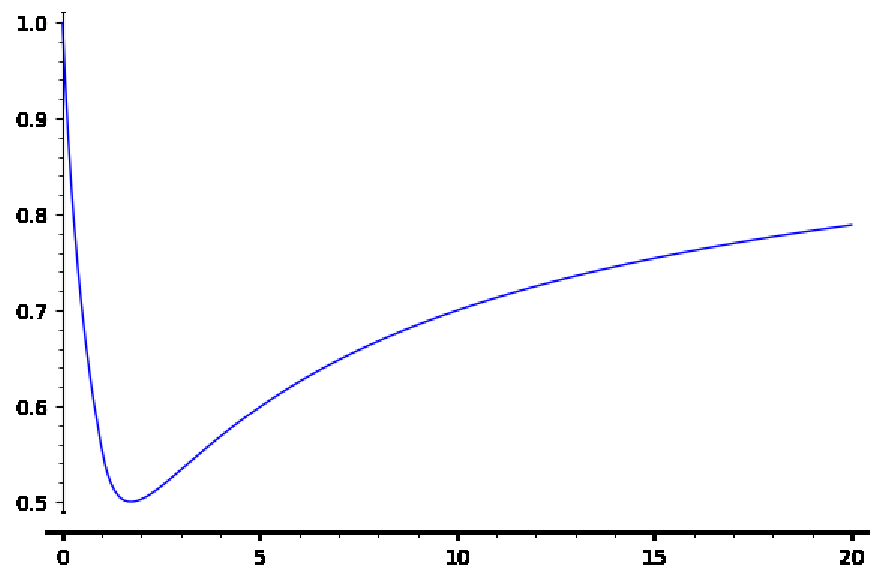}
\caption{$\mu(\triangle PQR)$
corresponding to $s = \frac{|RS|}{|QS|}$.}
\label{fig:isocelesmu}
\end{figure}

\vspace{2cm}

\noindent
Takeo Noda: Faculty of Science, Toho University, JAPAN\\
{\it E-mail address: noda@c.sci.toho-u.ac.jp}

\noindent
Shin-ichi Yasutomi: Faculty of Science, Toho University, JAPAN\\
{\it E-mail address: shinichi.yasutomi@sci.toho-u.ac.jp}
\end{document}